\documentclass[11pt]{article}
\usepackage{a4wide}
\usepackage{latexsym,stmaryrd}
\usepackage{amsmath,amssymb,amsthm,bm, enumerate}
\usepackage[utf8]{inputenc}
\usepackage{t1enc, lmodern}
\usepackage{appendix}
\usepackage{graphicx}
\usepackage[margin=1in]{geometry}
\usepackage{algorithm, algorithmicx, algpseudocode}
\usepackage{fancyhdr}
\usepackage{cancel}
\usepackage[bookmarks,colorlinks,linkcolor=blue]{hyperref}
\usepackage{ulem}
\usepackage[usenames,dvipsnames,svgnames,table]{xcolor}

\newcommand{\rset}{\mathbb{R}}
\newcommand{\R}{\mathbb{R}}
\newcommand{\PP}{\mathbb{P}}

\newcommand{\nset}{\mathbb{N}}

\newcommand{\ep}{\varepsilon}
\newcommand{\ind}{\mathbf{1}}

\newcommand{\e}{\mathbb{E}}

\newcommand{\p}{\mathbb{P}}

\newcommand{\cf}{\mathcal{F}}

\newcommand{\eps}{\varepsilon}

\newcommand{\ti}{\tilde}

\newcommand{\les}{\hspace{-2em}}
\newcommand{\equa}{\begin{eqnarray*}}
\newcommand{\tion}{\end{eqnarray*}}
\newcommand{\equal}{\begin{eqnarray}}
\newcommand{\tionl}{\end{eqnarray}}
\newcommand{\pull}{\hspace{-2em}}
\newcommand{\half}{\frac{1}{2}}

\newcommand{\pr}[1]{\left( #1 \right)}
\newcommand{\E}{\mathbb{E}}
\newcommand{\abs}[1]{\left|#1\right|}
\newcommand{\cG}{\mathcal{G}}
\newcommand{\less}{\hspace*{-4em}}

\theoremstyle{plain}
\newtheorem{thm}{Theorem}[section]
\newtheorem{lemme}[thm]{Lemma}

\newtheorem{prop}[thm]{Proposition}

\newtheorem{hypo}[thm]{Assumption}

\theoremstyle{definition}

\theoremstyle{plain}
\newtheorem{rem}[thm]{Remark}

\title{{\bf Random walk approximation of BSDEs with Hölder continuous terminal
  condition}\\  }
\author{Christel Geiss$^1$, Céline Labart$^2$, Antti Luoto$^3$ \\
{
}}

\date{}

\begin{document}
\allowdisplaybreaks
\maketitle
\begin{abstract}
In this paper we consider the random walk approximation of the solution
 of a Markovian BSDE whose terminal condition is a locally Hölder continuous
 function of the Brownian motion. We
 state the rate of the $L_2$-convergence of the approximated solution to the
 true one. The proof relies in part on growth and smoothness properties of the
 solution $u$ of the associated PDE. Here we improve existing results by showing some
 properties of the second derivative of $u$ in space.                 

\end{abstract}

{\bf Keywords :}  Backward stochastic differential equations, numerical scheme, random
  walk approximation, speed of convergence\\

{\bf MSC codes :} 65C30  60H35 60G50 65G99

{\noindent
\footnotetext[1]{Department of Mathematics and Statistics, P.O.Box 35 (MaD), FI-40014 University of Jyvaskyla, Finland \\ \hspace*{1.5em}{\tt christel.geiss{\rm@}jyu.fi}}
\footnotetext[2]{Univ. Grenoble Alpes, Univ. Savoie Mont Blanc, CNRS, LAMA,
  73000 Chambéry, France \\ \hspace*{1.5em}  {\tt celine.labart{\rm@}univ-smb.fr}}}
\footnotetext[3]{Department of Mathematics and Statistics, P.O.Box 35 (MaD), FI-40014 University of Jyvaskyla, Finland \\ \hspace*{1.5em}{\tt antti.k.luoto{\rm@}student.jyu.fi}}
\section{Introduction}
Let  $(\Omega, \cf, \PP)$ be a complete probability space carrying the standard
Brownian motion $B= (B_t)_{t \ge 0}$ and assume $(\cf_t)_{t\ge 0}$ is the
augmented natural filtration. We consider the following backward stochastic
differential equation (BSDE for short)
\begin{align}\label{BSDE}
Y_s&= g( B_T) + \int_s^T f(r,B_r,Y_r,Z_r)dr - \int_s^T Z_r dB_r,  \quad \quad 0 \leq s  \leq T,
\end{align}
where $f$  is Lipschitz continuous
and $g$ is a  locally $\alpha$-Hölder continuous and
polynomially bounded function (see \eqref{eq5}). In this paper we are
interested in the $L_2$-convergence of the numerical approximation of \eqref{BSDE} by using a random
walk. First results dealing with the numerical approximation of BSDEs  
date back to
the late 1990s. Bally (see \cite{Bally_97}) was the first to consider this problem
by introducing random discretization, namely the jump times of a Poisson
process. In his PhD thesis, Chevance (see \cite{Chevance_97}) proposed the following 
discretization
\begin{align*}
  y_k=\E(y_{k+1}+hf(y_{k+1})|\mathcal{F}^n_k), \qquad k=n-1,\cdots,0,\quad n \in \nset^*
\end{align*}
and proved the convergence of $(Y^n_t)_t:=(y_{[t/h]})_t$ to $Y$.
 At the same time,
Coquet, Mackevi\v{c}ius and Mémin \cite{CMM_99} proved the convergence of $Y^n$ by using convergence of filtrations, still
in the case of a generator independent from $z$. The general case ($f$ depends
on $z$, terminal condition $\xi \in L_2$) has been studied
by Briand, Delyon and Mémin (see \cite{Briand-etal}). In that paper the
authors define 
an approximated solution $(Y^n,Z^n)$  based on random walk and prove weak convergence to $(Y,Z)$ using convergence of
filtrations. We also refer to \cite{MaProMartTor}, \cite{MarMarTor},
\cite{MeminPengXu}, \cite{PengXu} for other numerical methods for BSDEs which use a 
random walk approach. The rate of
convergence of this method was left as an open problem. 

Introducing instead of random walk an approach based on the dynamic 
programming equation, Bouchard and Touzi in \cite{BT04} and Zhang in \cite{Z04}
managed to establish a rate of convergence. However, to be fully implementable, this
algorithm requires to have a good approximation of its associated conditional
expectation. For this, various methods have been developed (see \cite{GLW05},
\cite{CMT10}, \cite{CGT17}). Forward methods have also been introduced to
approximate \eqref{BSDE} : a branching diffusion method (see
\cite{HLTT13}), a multilevel Picard approximation (see \cite{WHJK17}) and Wiener
chaos expansion (see \cite{BL14}). Many extensions of \eqref{BSDE}
have also been considered : high order schemes (see \cite{Cha14},
\cite{CC14}), schemes for reflected BSDEs (see \cite{BP03}, \cite{CR16}), for
fully-coupled BSDEs (see \cite{Delarue_and_Menozzi}, \cite{BZ08}), for quadratic BSDEs (see
\cite{CR15}), for BSDEs with jumps (see \cite{GL16}) and for McKean-Vlasov BSDEs
(see \cite{A15}, \cite{CT15}, \cite{CCD17}).\\
 From a numerical point
of view, the random walk is of course not competitive with recent methods
listed above. We emphasize that the aim of this paper is to give the
convergence rate of the initial method based on random walk, which, to the best of our knowledge, has not been done so far.

As in \cite{Briand-etal}, let us introduce the following
approximation of $B$, based on a random walk:
$$B^n_{t} = \sqrt{h} \sum_{i=1}^{[t/h] }\varepsilon_i, \quad \quad 0 \leq t
\leq T,$$
where $h= \tfrac{T}{n}$ ($n \in \nset^*$) and $(\varepsilon_i)_{i=1,2, \dots}$ is a sequence of
i.i.d.~Rademacher random variables. Consider the following approximated solution $(Y^n,Z^n)$ of $(Y,Z)$
\begin{align}\label{discrete-BSDE-sums}
  Y^n_{t_k} &= g(B^n_T) + h\sum_{m=k}^{n-1}
              f(t_{m+1},B^n_{t_m},Y^n_{t_m},Z^n_{t_m}) -
              \sqrt{h}\sum_{m=k}^{n-1}Z^n_{t_m} \ep_{m+1}, \quad 0 \leq k \leq
              n-1.
\end{align}
The main result of our paper gives the rate of convergence in $L_2$-norm of
$Y^n_{v}-Y_{v}$ and $Z^n_v-Z_v$ for each $v \in [0,T)$ (see Theorem
\ref{the-Brownian-motion-result}). Basically, we get that the $L_2$-norm of
the error on $Y$ is of order $h^{\frac{\alpha}{4}}$ and the $L_2$-norm of
the error on $Z$ is of order $\frac{h^{\frac{\alpha}{4}}}{\sqrt{T-v}}.$ The proof of this result is based on several
ingredients. In particular, we need some estimates on the bound of the first and
second derivatives of the solution of the 
PDE associated to the BSDE \eqref{BSDE}. We establish these bounds in the case
of a forward backward SDE (FBSDE for short) whose terminal condition satisfies the H\"older continuity condition \eqref{eq5}. This
result extends Zhang \cite[Theorem 3.2]{ZhangII}.\\

The rest of the paper is organized as follows. Section \ref{sect:preliminaries}
introduces notations, assumptions and the representation for $Z$ and $Z^n$
based on the Malliavin weights. Section \ref{sect:main_result} states the rate
of convergence of the error on $Y$ and $Z$ in $L_2$-norm, which is the main
result of the paper. Section \ref{sect:numerics} presents numerical
simulations and Section \ref{sect:PDE_FBSDE} recalls some properties of 
Malliavin weights, of the regularity of solutions to FBSDEs with a locally Hölder continuous terminal condition function and 
states some properties of the solutions to the PDEs associated to these
FBSDEs.

\section{Preliminaries}\label{sect:preliminaries}
This section is dedicated to notations, assumptions and the representation of
$Z$ and $Z^n$ using the Malliavin weights. \smallskip

{\bf Notation:}
\begin{itemize}
\item $\mathcal{G}_k := \sigma(\varepsilon_i: 1 \leq i \leq k)$ and $\mathcal{G}_0 =\{\emptyset, \Omega\}.$  The associated discrete-time random walk $(B^n_{t_k})_{k=0}^n$ is
$(\mathcal{G}_k)_{k=0}^n$-adapted.
\item $\|\cdot\|_p:=\|\cdot\|_{L^p(\p)}$ for $p\ge 1$ and for $p=2$ simply
  $\|\cdot\|$.
constant.
\end{itemize}

\begin{hypo}\label{hypo3}\hfill
\begin{itemize}
\item   $g$ is locally H\"older continuous with order $\alpha \in  (0,1]$ and polynomially bounded ($p_0 \ge 0, C_g>0$)  in the following sense
  \begin{align} \label{eq5}
    \forall (x,y) \in \R^2, \quad |g(x)-g(y)|\le C_g(1+|x|^{p_0}+ |y|^{p_0})|x-y|^{\alpha}.
  \end{align}
\item The function   $ [0,T] \times \R^3: (t,x,y,z) \mapsto f (t,x,y,z)$ satisfies 
   \equal  \label{Lipschitz}
  |f(t,x,y,z) - f(t',x',y',z')| \le L_f(\sqrt{t-t'} + |x-x'| + |y-y'|+|z-z'|).
  \tionl   
\end{itemize}
\end{hypo}
Notice that \eqref{eq5} implies
\equal \label{Psi}
|g(x)| \le K(1+|x|^{p_0+1}) =:\Psi(x) .
\tionl

In the rest of the paper, the study of the error $(Y^n-Y,Z^n-Z)$ will either
rely on \eqref{discrete-BSDE-sums} or on its integral version:

\begin{align}\label{discrete-BSDE}
  Y^n_s &= g(B^n_T) + \int_{(s,T]} f(r,B^n_{r^-},Y^n_{r^-},Z^n_{r^-})d[B^n, B^n]_r - \int_{(s,T]} Z^n _{r^-} dB^n_r, \quad 0  \le s \leq T,
\end{align} 
where the backward equation \eqref{discrete-BSDE} arises from \eqref{discrete-BSDE-sums} by setting $Y^n_r:=Y^n_{t_m}$ and $Z_r^n :={Z}_{t_m}^n$  for $r \in
[t_m, t_{m+1}).$
For $n$ large enough, \eqref{discrete-BSDE} has a unique solution $({Y}^n,
{Z}^n),$  and $({Y}^n_{t_m}, {Z}^n_{t_m})_{m=0}^{n-1}$ is adapted to the
filtration $(\cG_m)_{m=0}^{n-1}$. Let us now introduce the Malliavin
representations for $Z$ and $Z^n$. They are the cornerstone of
our study of the error on $Z$.

\subsection{Representations for \texorpdfstring{$Z$}{Z} and \texorpdfstring{$Z^n$}{Zn}}

We will use the representation (see Ma and Zhang \cite[Theorem 4.2]{MaZhang})
\begin{align}\label{Z-representation}
   Z_t &= \e_{t} \left( g(B_T) N^t_T + \int_t^T f(s,B_s,Y_s,Z_s) N^t_s
     ds\right), \quad 0 \leq t \leq T,
\end{align}
where $\E_t[\cdot]=\E[\cdot|\mathcal{F}_t],$ and for all $s \in (t,T]$ we have
\equa
  N^t_s:=\frac{B_s-B_t}{s-t}.  \tion

\begin{lemme} Suppose that Assumption \ref{hypo3} holds. Then the process $Z^n$ given by \eqref{discrete-BSDE} has the representation 
\begin{align}  \label{Z-n-representation}
&Z^n_{t_k} = \E_{k} \left (g(B_T^n) \frac{B^n_{t_n} - B^n_{t_k}}{t_n - t_k}\right ) + \E_{k}  \left (h \sum_{m={ k+1}}^{n-1} f( t_{m+1}, B^n_{t_m}, Y^n_{t_m}, 
Z^n_{t_m})\frac{B^n_{t_m} - B^n_{t_k}}{t_m - t_k} \right )
\end{align}
for $k = 0,1, \dots, n-1,$ where $ \E_k [\, \cdot \,] := \E[\, \cdot \, | \mathcal{G}_k].$
\end{lemme}
\begin{proof}
We multiply equation (\ref{discrete-BSDE-sums}) by $\varepsilon_{k+1}$ and take the conditional expectation with respect to 
$\mathcal{G}_k$. Since $(Y^n_{t_k},Z^n_{t_k})$ is $\mathcal{G}_k$-measurable, it holds for $0 \leq k \leq n-1$ that
\begin{align}\label{cond-exp-three-parts}
& \E_{k} \left (Y^n_{t_k} \varepsilon_{k+1}\right ) \notag\\  & = \E_{k} \left (g(B^n_T) \varepsilon_{k+1} \right ) + h \E_{k} \left (\sum_{m=k}^{n-1} f(t_{m+1},B^n_{t_m}, Y^n_{t_m}, Z^n_{t_m}) \varepsilon_{k+1} \right )  - \sqrt{h} \E_{k}\left ( \sum_{m=k}^{n-1} Z^n_{t_m} \varepsilon_{m+1}\varepsilon_{k+1}\right )  \nonumber\\
&= \sqrt{h} \E_k \left ( g(B^n_T) \frac{B^n_{t_n} - B^n_{t_k}}{t_n-t_k} \right )  + h^{3/2} \sum_{m=k+1}^{n-1}\E_{k} \left ( f(t_{m+1}, B^n_{t_m},Y^n_{t_m}, Z^n_{t_m}) \frac{B^n_{t_m} - B^n_{t_k}}{t_m - t_k}\right )  - \sqrt{h} Z^n_{t_k},
\end{align}
where the l.h.s.~is equal to zero. Indeed, for $m \geq k + 1$, we have 
$$\E_{k}  (Z^n_{t_m}  \varepsilon_{m+1}  \varepsilon_{k+1}) = \E_{k} ( Z^n_{t_m} \varepsilon_{k+1} \E_{m}  \varepsilon_{m+1} )  = 0,$$
and for $m = k$ it holds $\E_{k} (Z^n_{t_k} \varepsilon_{k+1}^2) = Z^n_{t_k}.$ Moreover, the fact that $B_T^n = \sqrt{h} \sum_{m=0}^{n-1} \varepsilon_{m+1}$, 
where $(\varepsilon_{m})_{m=1,2 \dots}$ 
are i.i.d., yields
\begin{align*}
\E_k  \left (  g(B^n_T) \ep_{k+1} \right ) = \E_k \left ( g(B^n_T) \sum_{m=k}^{n-1} \frac{\ep_{k+1}}{n-k}\right ) = \E_k \left ( g(B^n_T) \sum_{m=k}^{n-1} \frac{\ep_{m+1}}{n-k}\right ) = \sqrt{h} \E_k \left ( g(B^n_T) \frac{B^n_{t_n} - B^n_{t_k}}{t_n-t_k} \right ).
\end{align*}
Similarly, for $m \geq k+1$, we get (using \cite[Proposition 5.1]{Briand-etal}, where it is stated that both $Y^n_{t_m}$ and $Z^n_{t_m}$ 
can be represented as  functions of $t_m$ and $B^n_{t_m}$)
\begin{align*}
\E_{k}  \left (  f(t_{m+1}, B^n_{t_m},Y^n_{t_m}, Z^n_{t_m}) \varepsilon_{k+1} \right ) = \sqrt{h} \E_{k} \left (  f(t_{m+1},B^n_{t_m}, Y^n_{t_m}, Z^n_{t_m}) \frac{B^n_{t_m} - B^n_{t_k}}{t_m - t_k} \right ).
\end{align*}
It remains to divide  \eqref{cond-exp-three-parts} by $\sqrt{h}$ and rearrange.
\end{proof}

\section{Main result}\label{sect:main_result}
This section is devoted to the main result of the paper: the rate of the $L_2$-convergence
of $(Y^n,Z^n)$ to $(Y,Z)$. The proof will rely on the fact that the random walk $B^n$ can be constructed from the Brownian motion $B$ by Skorohod embedding.
Let    $\tau_0 :=0$ and define
\[
   \tau_k := \inf\{ t> \tau_{k-1}: |B_t-B_{\tau_{k-1}}| = \sqrt{h} \}, \quad k \ge 1.
\]
Then  $(B_{\tau_{k}}   -B_{\tau_{k-1}})_{k=1}^\infty$ is a sequence of i.i.d. random variables with 
\[\PP(  B_{\tau_{k}}   -B_{\tau_{k-1}} = \pm \sqrt{h})= \tfrac{1}{2},
\]
which means that $\sqrt{h} \ep_k  \stackrel{d}{=}  B_{\tau_{k}}
-B_{\tau_{k-1}}.$ We will  use  this random walk  for our approximation, i.e. we will require
\equal \label{embedded}
   B^n_{t} =  \sum_{k=1}^{[t/h]}  (B_{\tau_{k}}-B_{\tau_{k-1}}), \quad \quad 0 \leq t \leq T.
\tionl
Properties satisfied by $\tau_k$ and $B_{\tau_k}$ are stated in
Lemma \ref{B-difference}. We will denote by $ \e_{\tau_k}$ the conditional expectation w.r.t. $\cf_{\tau_k}.$

\begin{thm} \label{the-Brownian-motion-result} 
 Let   Assumption \ref{hypo3} hold. If  $B^n$ satisfies  \eqref{embedded} then 
we have  (for sufficiently large  $n$) that
 \begin{align*}
&\e |Y_v- Y^n_v |^2 \le    C_0 h^{ \frac{\alpha}{2}}  \quad \text{for}  \quad  v \in [0,T), \\
 & \e |Z_{v} - Z^n_{v}|^2 \le  C_0 \frac{h^{ \frac{\alpha}{2}}}{T-t_k} +  C_1 \frac{h^{ \frac{\alpha}{2}} }{(T-v)^{1-\frac{\alpha}{2}}}\ind_{v \neq t_k } \quad \text{for} \quad v \in [t_k,t_{k+1}), \,\, k=0,...,n-1,
  \end{align*}
where we have the dependencies   $ C_0 = C(T, p_0, L_f, C_g, C^y_{\ref{difference-estimates for Y and Z}},C^z_{\ref{difference-estimates for Y and Z}},K_f, c_{\ref{thm1}}, \alpha),$ 
$C_1=C(T,p_0, C^z_{\ref{difference-estimates for Y and Z}}, \alpha)$ and   $K_f: = \sup_{0\le t\le T} |f(t,0,0,0)|.$ 
\end{thm}

\begin{rem}
Theorem \ref{the-Brownian-motion-result}  implies that 
\equa   \sup_{ v \in [0,T)}  \e |Y_v- Y^n_v |^2 \le    C_0 h^{ \frac{\alpha}{2}} \quad \text{ and }  
\quad  \e \int_0^T   |Z_{v} - Z^n_{v}|^2 dv \le C(C_0,C_1, \beta ) \, h^{\beta}  \quad \text{ for } \beta \in (0, \tfrac{\alpha}{2}).
\tion
\end{rem}

\begin{proof}[ Proof of Theorem \ref{the-Brownian-motion-result}]
Let $u : [0,T)\times \R \to \R$ be the solution of the  PDE associated to
\eqref{BSDE}. Since by Theorem \ref{thm1}
\equa
   Y_s = u(s, B_s), \quad  Z_s = u_x(s, B_s), \quad a.s.
\tion
we introduce 
\equa
F(s,x) := f(s, x, u(s, x), u_x(s, x)),
\tion
so that   $F(s,B_s)  =  f(s, B_s, Y_s, Z_s).$
We first give some properties satisfied by $F$.
   \begin{lemme}\label{F-properties} If  Assumption \ref{hypo3} holds then $F$ is a Lipschitz continuous  and
 polynomially bounded function in $x$ :
    \begin{align*}
     |F(t,x_1) -F(t,x_2)|& \le C( T, L_f, c^{2,3}_{\ref{thm1}}) (1+|x_1|^{p_0+1} +
    |x_2|^{p_0+1})\frac{|x_1 - x_2|}{(T-t)^{1-\frac{\alpha}{2}}},\\
      |F(t,x)|& \le  C( T, L_f,  c^{1,2}_{\ref{thm1}},K_f) \frac{\Psi(x)}{(T-t)^{ \frac{1-\alpha}{2}}},
    \end{align*}
    where $\Psi(x)$ is given in \eqref{Psi}.
\end{lemme}
\begin{proof}[Proof of Lemma \ref{F-properties}]
  Thanks to the mean value theorem and Theorem \ref{thm1}-(ii-c) and (iii-b) we have  for $x_1, x_2 \in \R$  that there exist 
$\xi_1, \xi_2 \in [\min\{x_1,x_2\},\max\{x_1,x_2\}] $ such that
\equa 
|F(t,x_1) -F(t,x_2)| &=& |f(t, x_1, u(t, x_1), u_x(t, x_1)) - f(t, x_2, u(t, x_2), u_x(t, x_2))| \notag \\
& \le & L_f (  |x_1 - x_2|  + |u(t, x_1)-u(t, x_2)| + |u_x(t, x_1)- u_x(t, x_2)|) \notag\\
& \le & L_f  \left  ( 1 +  \frac{c^2_{\ref{thm1}} \Psi(\xi_1)}{(T-t)^{\frac{1-\alpha}{2}}}  +    \frac{ c^3_{\ref{thm1}} \Psi(\xi_2)}{(T-t)^{1-\frac{\alpha}{2}}} \right )|x_1-x_2|  \notag\\
 & \le & C(T,L_f, c^{2,3}_{\ref{thm1}}) (1+|x_1|^{p_0+1} + |x_2|^{p_0+1})\frac{|x_1 - x_2|}{(T-t)^{1-\frac{\alpha}{2}}}.
 \tion
 The second inequality can be shown similarly.
\end{proof}
For  the estimate of $\e |Y_{t_k}-  Y^n_{t_k} |^2$ we will use \eqref{BSDE} and \eqref{discrete-BSDE-sums}: Since  $Y^n_{t_k}$ is $\cf_{\tau_k}$-measurable we have 
\equal \label{y-difference} 
 \|Y_{t_k}-  Y^n_{t_k} \| &\le  &  \| \e_{t_k} g(B_T)  -  \e_{\tau_k} g(B^n_T) \|   \notag \\
 &&+
 \left \| \e_{t_k}\int_{t_k}^T f(s,B_s,Y_s,Z_s)ds  - h \e_{\tau_k}\sum_{m=k}^{n-1} f(t_{m+1}, B^n_{t_m},Y^n_{t_m},Z^n_{t_m})  \right \|.
\tionl
We  frequently express conditional expectations with the help of an independent copy of $B$ denoted by  $\ti B,$ for example  
$\e_{t} g(B_T) = \ti \e g(B_t + \ti B_{T -t}).$ 

By \eqref{eq5} and Lemma \ref{B-difference},
\equal \label{terminal}
\| \e_{t_k} g(B_T)  -  \e_{\tau_k} g(B^n_T) \|^2 &=& \e|\ti \e g(B_{t_k} + \ti B_{T - t_k}) - \ti \e g(B_{\tau_k} + \ti B_{\ti \tau_{n-k}})  |^2  \notag \\
&\le&  ( \E \ti  \e  (\Psi_1)^4 )^\half  ( \E \ti  \e   | B_{t_k} -   B_{\tau_k}    +  \ti B_{T - t_k} -  \ti B_{\ti \tau_{n-k}} |^{4 \alpha})^\half   \notag \\
&\le& C(C_g,T,p_0)(  (\E  | B_{t_k} -   B_{\tau_k}  |^{4 \alpha} )^\half  +  (\e|   B_{T- {t_k}} - B_{\tau_{n-k}}  |^{4 \alpha})^\half )   \notag\\
&\le&   C(C_g,T,p_0) h^\frac{\alpha}{2},
\tionl
where $\Psi_1:=C_g(1+|B_{t_k} + \ti B_{T - t_k}|^{p_0}+ |B_{\tau_k} + \ti B_{\ti \tau_{n-k}} |^{p_0}).$ 
To estimate the other term in \eqref{y-difference}  we  consider the decomposition
\equa
&& \hspace{-2em} \e_{t_k} f(s,B_s,Y_s,Z_s)-\e_{\tau_k}f(t_{m+1}, B^n_{t_m},Y^n_{t_m},Z^n_{t_m}) \\
&=&\!\!\! (\e_{t_k} f(s, B_s,Y_s,Z_s)- \e_{t_k}f(t_m, B_{t_m},Y_{t_m},Z_{t_m}) ) +( \e_{t_k}F(t_m,B_{t_m}) -\e_{\tau_k}F(t_m,B_{\tau_m})) \\
&&\!\!\! + (\e_{\tau_k}F(t_m,B_{\tau_m}) -\e_{\tau_k} F(t_m,B_{t_m})) 
 +(\e_{\tau_k} f(t_m, B_{t_m}, Y_{t_m},Z_{t_m}) - \e_{\tau_k}f(t_{m+1}, B^n_{t_m},Y^n_{t_m},Z^n_{t_m})) \\
&=:& D_1(s,m) + D_2(m)+...+ D_4(m)
\tion
so that 
\equa
&& \pull \left \|    \e_{t_k}\int_{t_k}^T f(s,B_s,Y_s,Z_s)ds -   h \e_{\tau_k}\sum_{m=k}^{n-1} f(t_{m+1}, B^n_{t_m},Y^n_{t_m},Z^n_{t_m}) \right \| 
 \notag\\&& 
 \le \sum_{m=k}^{n-1} \left (  \left \|  \int_{t_m}^{t_{m+1}}  D_1(s,m)  ds\right \|  +  h \sum_{i=2}^4  \|D_i(m) \| \right).
\tion
For  $D_1$ we have by Theorem \ref{difference-estimates for Y and Z} that
\equal \label{D1}
   \|    D_1(s,m)  \| 
&\le& L_f ( \sqrt{s-t_m} + \|B_s-B_{t_m}\|+   \|Y_s -Y_{t_m}\|+  \|Z_s -Z_{t_m}\|  ) \notag \\
&\le& C(T, L_f,C^y_{\ref{difference-estimates for Y and Z}},C^z_{\ref{difference-estimates for Y and Z}}, p_0) \,  (T-s)^{\frac{\alpha-2}{2}} \,h^\half ,
\tionl where the last inequality follows from $\|B_s-B_{t_m}\|= \sqrt{s-t_m} \le h^\half$  for $s \in [t_m, t_{m+1}]$  and 
\equa
 \|Y_s -Y_{t_m}\|+  \|Z_s -Z_{t_m}\| &\le&  ( \e\Psi(B_{t_m})^2)^\half  \left (  C^y_{\ref{difference-estimates for Y and Z}} \left (  \int_{t_m}^s (T-r)^{\alpha-1}  dr \right )^\half 
  +  C^z_{\ref{difference-estimates for Y and Z}} \left (   \int_{t_m}^s  (T-r)^{\alpha-2}  dr \right )^\half  \right )\\
 &\le&  C( T, C^y_{\ref{difference-estimates for Y and Z}}, C^z_{\ref{difference-estimates for Y and Z}}, p_0) \sqrt{s-t_m} ((T-s)^{\frac{\alpha-1}{2}}+ (T-s)^{\frac{\alpha-2}{2}}). 
 \tion
We bound $D_2$  using Lemma \ref{F-properties} and Lemma
\ref{B-difference}. Similar to \eqref{terminal} we conclude  (setting $\Psi_2:= 1+|B_{t_k}+  \ti B_{t_{m-k}}|^{p_0+1}+|B_{\tau_k} + \ti B_{\ti \tau_{m-k}}|^{p_0+1}$)
that
\equa
 \|D_2(m) \| &=& \left ( \e \left |\e_{t_k}F(t_m,B_{t_m}) -\e_{\tau_k}F(t_m,B_{\tau_m})\right|^2  \right )^\half \\
 &\le&   C( T, L_f, c^{2,3}_{\ref{thm1}})   (\e \ti \e \Psi_2^4 )^\frac{1}{4}\frac{1}{(T-t_m)^{1-\frac{\alpha}{2}}} \left (t_k  h +  t_{m-k} h \right)^\frac{1}{4}\\
&\le& C(T,p_0,L_f, c^{2,3}_{\ref{thm1}}) \frac{1}{(T-t_m)^{1-\frac{\alpha}{2}} } h^{\frac{1}{4}}.
\tion
For $D_3$ we apply again Lemma \ref{F-properties} and Lemma \ref{B-difference},
\equa
 \|D_3(m) \| \le  \|F(t_m,B_{t_m}) -F(t_m,B_{\tau_m})\|
&\le&  C(T, L_f,c^{2,3}_{\ref{thm1}})  \frac{1}{(T-t_m)^{1-\frac{\alpha}{2}}} \| \Psi_3\, |B_{t_m} -B_{\tau_m}|\| \\
&\le& C(T, p_0,L_f,c^{2,3}_{\ref{thm1}})  \frac{1}{(T-t_m)^{1-\frac{\alpha}{2}}}  h^\frac{1}{4},
\tion
where  $\Psi_3 := 1 + |B_{t_m}|^{p_0+1} + |B_{\tau_{m}}|^{p_0+1}$. 
For the last term $D_4$ we get
\equa
 \|D_4(m) \| &\le & 
 L_f (  h^\half + \| B_{t_m} - B^n_{t_m}\|    +  \| Y_{t_m} - Y^n_{t_m}\|+  \| Z_{t_m} -Z^n_{t_m}\| ).
\tion
Finally, using  the  estimates for the terms  $D_1(s,m), D_2(m),...,D_4(m)$ we arrive at
 \equal \label{Y-difference-estimate}
  \| Y_{t_k} - Y^n_{t_k}\|  
    &\le &    C(C_g,T,p_0){h^\frac{\alpha}{4} } + C( T, L_f,C^y_{\ref{difference-estimates for Y and Z}},C^z_{\ref{difference-estimates for Y and Z}}, p_0) \,
       h^\half \int_{t_k}^T  (T-s)^{\frac{\alpha-2}{2}}  ds   \notag \\
&& +  C(T,p_0, L_f, c^{2,3}_{\ref{thm1}})h^{\frac{1}{4}} \sum_{m=k}^{n-1} \frac{h}{(T-t_m)^{1-\frac{\alpha}{2}}}  +  h L_f \sum_{m=k}^{n-1}(  \| Y_{t_m} - Y^n_{t_m}\| +  \| Z_{t_m} -Z^n_{t_m}\| ) \notag\\
&\le &  C(C_g,T,p_0,L_f, c^{2,3}_{\ref{thm1}},C^y_{\ref{difference-estimates for Y and Z}},C^z_{\ref{difference-estimates for Y and Z}})  h^{\frac{\alpha}{4}}
+  h L_f \sum_{m=k}^{n-1}(  \| Y_{t_m} - Y^n_{t_m}\| +  \| Z_{t_m} -Z^n_{t_m}\|).
\tionl

For $\| Z_{t_k} -Z^n_{t_k}\|$ we exploit the representations \eqref{Z-representation} and \eqref{Z-n-representation} and estimate
\equa
  \| Z_{t_k} -Z^n_{t_k}\|   &\le & \frac{1}{T-t_k} \| \e_{t_k}  g(B_T)(B_T-B_{t_k}) -\e_{\tau_k} g(B_{\tau_n}) (B_{\tau_n}-B_{\tau_k}  )  \| \\
&& +  \Big  \|\e_{t_k}  \left ( \int_{t_{k+1}}^T f(s, B_s, Y_s,Z_s) \frac{B_s- B_{t_k} }{s-t_k} ds \right) \\
&&  \quad \quad -\e_{\tau_k}  \left ( h \sum_{m= k+1}^{n-1} f(t_{m+1}, B^n_{t_m},Y^n_{t_m},Z^n_{t_m})\frac{B^n_{t_m} - B^n_{t_k}}{t_m - t_k} \right )  \Big \| \\
&&  +  \Big  \|\e_{t_k}\int_{t_k}^{t_{k+1}} f(s, B_s, Y_s,Z_s) \frac{B_s- B_{t_k} }{s-t_k} ds \Big  \|.
\tion
Then, similar to  \eqref{terminal}, we have for the terminal condition  by  Lemma \ref{B-difference} that 
\equa
 && \les  \| \e_{t_k}  [g(B_T)(B_T-B_{t_k})] -\e_{\tau_k} [g(B_{\tau_n}) (B_{\tau_n}-B_{\tau_k}  )]  \| \\ 
  &= &  \|  \ti \e [g(B_{t_k} + \ti B_{T-t_k} ) -  g(B_{t_k})] (\ti B_{T-t_k} -\ti B_{\ti \tau_{n-k} })  +   \ti \e  [ g(B_{t_k} + \ti B_{T-t_k} ) -  g(B_{\tau_k}    +\ti B_{\ti \tau_{n-k} } )] \ti B_{\ti \tau_{n-k} }\|\\
  &\le &   C(C_g,T,p_0) h^\frac{1 }{4}  (T-t_k)^{\frac{\alpha }{2} +\frac{1 }{4}} 
  +  C(C_g,T,p_0) h^\frac{\alpha}{4} (T-t_k)^\half \le  C(C_g,T,p_0) h^\frac{\alpha }{4}  (T-t_k)^\frac{1}{2}.
\tion 
Here we have used that $\ti \e [g(B_{t_k}) (\ti B_{T-t_k} -\ti B_{\ti \tau_{n-k} }) ]=0$. The term $\ti \e [g(B_{t_k} + \ti B_{T-t_k} ) -  g(B_{t_k})] (\ti B_{T-t_k} -\ti B_{\ti \tau_{n-k}} )$ provides us with the factor $ (T-t_k)^{\frac{\alpha }{2}}((T-t_k)h)^{\frac{1}{4}}.$ 
For the next term of the estimate of $\| Z_{t_k} -Z^n_{t_k}\| $ we use for $s \in [t_m, t_{m+1})$,  where $m \geq k+1$,  the decomposition 
\equa
&& \pull \pull \frac{\e_{t_k}f(s, B_s ,Y_s,Z_s) (B_s- B_{t_k}) }{s-t_k} -\frac{\e_{\tau_k}f(t_{m+1}, B^n_{t_m}, Y^n_{t_m},Z^n_{t_m})(B^n_{t_m}-B^n_{t_k})}{t_m-t_k} \\
& = &\frac{\e_{t_k}f(s,B_s,Y_s,Z_s) (B_s- B_{t_k}) }{s-t_k} -\frac{\e_{t_k}f(t_m, B_{t_m},  Y_{t_m},Z_{t_m})(B_{t_m}-B_{t_k})}{t_m-t_k} \\
&& +\frac{\e_{t_k}F(t_m,B_{t_m})(B_{t_m}-B_{t_k})}{t_m-t_k} - \frac{\e_{\tau_k}F(t_m,B_{\tau_m})(B_{\tau_m}-B_{\tau_k})}{t_m-t_k} \\
&& + \e_{\tau_k}\left [ [F(t_m,B_{\tau_m}) - F(t_m,B_{t_m})] \frac{B_{\tau_m}-B_{\tau_k}}{t_m-t_k} \right ] \\
&& + \e_{\tau_k} \left [ [f(t_m, B_{t_m}, Y_{t_m},Z_{t_m})- f(t_{m+1}, B^n_{t_m}, Y^n_{t_m},Z^n_{t_m})]\frac{B^n_{t_m}-B^n_{t_k}}{t_m-t_k} \right ]  \\
&=:& T_1(s,m)+T_2(m) +...+ T_4(m).
\tion
 Then by the conditional H\"older inequality and by \eqref{D1} as well as by  Lemma \ref{F-properties}  we have
\equa
 \|T_1(s,m)\| &\le &\| D_1(s,m)\|\frac{ \|B_s- B_{t_k}\|}{s-t_k}  +\| f(t_m, B_{t_m}, Y_{t_m},Z_{t_m})\| \left \| \frac{ B_s- B_{t_k}}{s-t_k} - \frac{B_{t_m}-B_{t_k}}{t_m-t_k} \right \| \\
& \le & 
C( T, L_f,C^y_{\ref{difference-estimates for Y and Z}},C^z_{\ref{difference-estimates for Y and Z}},  p_0) \,   (T-s)^{\frac{\alpha-2}{2}}  \, \frac{h^\half}{ \sqrt{s-t_k}}\ \\
&& +  C( T, L_f,  c^{1,2}_{\ref{thm1}},K_f) \frac{(
    \e\Psi(B_{t_m})^2)^\half}{(T- t_m)^{\frac{1-\alpha}{2}}}  \\ && \times \left ( \frac{\|B_s- B_{t_m}\|}{s-t_k} + \|B_{t_m}-B_{t_k}\| \left |\frac{1}{s-t_k} - \frac{1}{t_m-t_k}  \right |\right ) \\
& \le &  
C(T, L_f, K_f, C^y_{\ref{difference-estimates for Y and Z}},C^z_{\ref{difference-estimates for Y and Z}},  c^{1,2}_{\ref{thm1}}, p_0)  (T-s)^{\frac{\alpha-2}{2}} \frac{h^\frac{1}{4} }{(s-t_k)^\frac{3}{4}}.
\tion
Indeed, 
\equa
\frac{\|B_s- B_{t_m}\|}{s-t_k} + \|B_{t_m}-B_{t_k}\| \left |\frac{1}{s-t_k} - \frac{1}{t_m-t_k} \right | & \le& \frac{\sqrt{s-t_m}}{s-t_k} + 
\frac{\sqrt{t_m-t_k} (s-t_m)}{(s-t_k)(t_m-t_k)}  
\le C  \frac{ h^\frac{1}{4} }{(s-t_k)^\frac{3}{4} },
\tion
where the last inequality follows from $s-t_m \le t_{m+1} -t_m = h$ and $ h \le t_m -t_k \le  s-t_k.$ We estimate $T_2$ 
with the help of Lemma \ref{F-properties} and  Lemma \ref{B-difference}  as follows :   
\equa
\|T_2(m)\|& \le &  \| \widehat D_2(m)\| \frac{ \|B_{t_m}-B_{t_k}\| }{t_m -t_k}  
 +    \|F(t_m,B_{\tau_m}) \|\frac{  \|B_{t_{m-k}} -B_{\tau_{m-k}} \|}{t_m -t_k}\\
&\le&C(T, p_0,L_f, K_f, c_{\ref{thm1}}) \frac{1}{(T-t_m)^{1-\frac{\alpha}{2}} }  \frac{h^\frac{1}{4}}{ (t_m -t_k)^\frac{3}{4}}.
\tion
  Here $\widehat D_2(m) := ( \ti \E |F(t_m, B_{t_k}+  \ti B_{t_{m-k}}) - F(t_m, B_{\tau_k}+ \ti B_{\ti \tau_{m-k}})|^2 )^{\frac{1}{2}}$ which can be 
estimated as $D_2(m).$
For $T_3$ the conditional H\"older inequality and Lemma \ref{B-difference} yield \equa
 \|T_3(m)\|&\le &   \|\widehat D_3(m)\| \left\|  \frac{B_{\tau_m}-B_{\tau_k}}{t_m-t_k}  \right\|  \le  C(T, p_0,L_f,c^{2,3}_{\ref{thm1}})  \frac{1}{(T-t_m)^{1-\frac{\alpha}{2}}}
 \frac{h^\frac{1}{4}}{ (t_m-t_k)^\frac{1}{2}},
\tion
where $\widehat D_3(m):=F(t_m, B_{\tau_m}) - F(t_m, B_{t_m})$ is estimated as $D_3(m).$
Finally,  \equa
 \|T_4(m)\|
 &\le &  L_f  ( h^\half + \|  B_{t_m} -B^n_{t_m} \|+ \|  Y_{t_m} -Y^n_{t_m} \| +  \| Z_{t_m} - Z^n_{t_m} \|)\frac{1}{ \sqrt{t_m-t_k}} .
 \tion
 For the estimate of   $ \Big  \| \e_{t_k}\int_{t_k}^{t_{k+1}} f(s, B_s, Y_s,Z_s) \frac{B_s- B_{t_k} }{s-t_k} ds  \Big \|$   one notices that by  the conditional H\"older inequality,
 \equa 
 \| \e_{t_k} f(s, B_s, Y_s,Z_s)  \tfrac{B_s- B_{t_k} }{s-t_k} \|&=& \| \e_{t_k} [(f(s, B_s, Y_s,Z_s) -  f(s, B_{t_k}, Y_{t_k},Z_{t_k})) \tfrac{B_s- B_{t_k} }{s-t_k}]\|\\
&\le &  \|f(s, B_s, Y_s,Z_s) -  f(s, B_{t_k}, Y_{t_k},Z_{t_k}) \|  \frac{1}{ \sqrt{s-t_k}} \\
&\le & C(T, L_f,C^y_{\ref{difference-estimates for Y and Z}},C^z_{\ref{difference-estimates for Y and Z}}, p_0) \, (T-s)^{\frac{\alpha-2}{2}} \, \frac{h^\half }{ \sqrt{s-t_k}}, 
\tion 
where  the last inequality follows in the same way  as in \eqref{D1}. Consequently,  we have
\equa
 \| Z_{t_k} -Z^n_{t_k}\| &\le &  \frac{ C(C_g,T,p_0)}{(T-t_k)^\frac{1}{2}}\, h^\frac{\alpha}{4}+  
 C(T, L_f, K_f, C^y_{\ref{difference-estimates for Y and Z}},C^z_{\ref{difference-estimates for Y and Z}},  c^{1,2}_{\ref{thm1}}, p_0) \int_{t_k}^T 
  \frac{ds  }{  (T-s)^{1-\frac{\alpha}{2}}    (s-t_k)^\frac{3}{4}} \, h^\frac{1}{4} \\
&&+   C(T,p_0,L_f, K_f,  c_{\ref{thm1}})  \, h \sum_{m=k+1}^{n-1}      \frac{1}{(T-t_m)^{1-\frac{\alpha}{2}} }  \frac{h^\frac{1}{4}}{ (t_m -t_k)^\frac{3}{4}}   \\\\
&&+   L_f h\sum_{m=k+1}^{n-1}  ( \|  B_{t_m} -B^n_{t_m} \|+ \|  Y_{t_m}
-Y^n_{t_m} \| +  \| Z_{t_m} - Z^n_{t_m} \|)\frac{1}{ \sqrt{t_{m-k}}}.
\tion
Lemma \ref{beta-function} enables  to bound the second and third term of the r.h.s. by $C
\frac{h^{\frac{1}{4}}}{(T-t_k)^{\frac{3}{4}-\frac{\alpha}{2}}}B(\frac{\alpha}{2},\frac{1}{4})$, which is
bounded by
$C\frac{h^{\frac{\alpha}{4}}}{(T-t_k)^{\half-\frac{\alpha}{4}}}$. Thus we get
\equa
\| Z_{t_k} -Z^n_{t_k}\| &\le &  \frac{ C_0 \, h^\frac{\alpha}{4} }{(T-t_k)^\frac{1}{2}}+   L_f h \sum_{m=k+1}^{n-1}  (\|  Y_{t_m} -Y^n_{t_m} \| +  \| Z_{t_m} - Z^n_{t_m} \|)\frac{1}{ \sqrt{t_{m-k}}}.
\tion
Then we use \eqref{Y-difference-estimate} and the above estimate to get
\equa
 \| Y_{t_k} - Y^n_{t_k}\| + \| Z_{t_k} -Z^n_{t_k}\| &\le &  \frac{ C_0 \, h^\frac{\alpha}{4} }{(T-t_k)^\frac{1}{2}} 
  +    C(L_f) \,h \sum_{m=k+1}^{n-1}  (\|  Y_{t_m} -Y^n_{t_m} \| +  \| Z_{t_m} - Z^n_{t_m} \|)\frac{1}{ \sqrt{t_{m-k}}}.
\tion
If  this inequality is iterated,  one  gets  a shape where the Gronwall lemma applies.  Indeed, setting
$a_m :=  (\|  Y_{t_m} -Y^n_{t_m} \| +  \| Z_{t_m} - Z^n_{t_m} \|)$  one has to consider the double sum
\equa
\sum_{m=k+1}^{n-1} \left (   \sum_{l=m+1}^{n-1}  a_l    \frac{h}{ \sqrt{t_{l-m}}} \right )   \frac{h}{ \sqrt{t_{m-k}}}  
&=& h \sum_{l=k+1}^{n-1} \left (   \sum_{m=k+1}^{l-1}   \frac{h}{ \sqrt{t_{m-k}} \sqrt{t_{l-m} }}    \right ) a_l   
\le  Ch \sum_{l=k+1}^{n-1}   a_l  .
\tion
Consequently,
\equa
 \| Y_{t_k} - Y^n_{t_k}\| + \| Z_{t_k} -Z^n_{t_k}\| &\le &  \frac{  C_0 \, h^\frac{\alpha}{4} }{(T-t_k)^\frac{1}{2}} 
 \tion
which gives the bound on the error on $Z$. Moreover, \eqref{Y-difference-estimate} yields
\equa
  \| Y_{t_k} - Y^n_{t_k}\|
    &\le &    C_0 \, h^{  \frac{\alpha}{4}} .
\tion

If $v \in [t_k,t_{k+1}),$ we have by Theorem  \ref{difference-estimates for Y and Z}  that
\begin{align*}
  & \|Y_v- Y^n_v \| \le   \|Y_v- Y_{t_k} \| +    \|Y_{t_k}- Y^n_{t_k} \|  \le
      C(C^y_{\ref{difference-estimates for Y and Z}},  T, p_0 ) \,
   \left ( \int_{t_k}^v (T-r)^{\alpha -1}dr \right )^\half   +     \|Y_{t_k}- Y^n_{t_k} \|,\\
 &\|Z_v- Z^n_v \| \le \|Z_v- Z_{t_k} \| +  \|Z_{t_k}- Z^n_{t_k} \|  \le 
   C(C^z_{\ref{difference-estimates for Y and Z}}, T, p_0)
  \left ( \int_{t_k}^v (T-r)^{\alpha -2}dr\right)^\half +  \|Z_{t_k}- Z^n_{t_k} \|,
\end{align*}
 where

$$  \int_{t_k}^v (T-r)^{\alpha -1}dr  \le  \frac{1}{\alpha}  (v-t_k)^\alpha   \le  \frac{1}{\alpha}  h^\alpha $$ 
 and  
 \begin{align*}
\int_{t_k}^v (T-r)^{\alpha - 2} dr & \leq \frac{1}{(T-v)^{1-\frac{\alpha}{2}}} \int_{t_k}^v (T-r)^{\frac{\alpha}{2} -1} dr\leq \frac{1}{(T-v)^{1-\frac{\alpha}{2}}} \frac{2}{\alpha} (v-t_k)^{\frac{\alpha}{2}} 
 \leq \frac{2}{\alpha}\frac{h^{\frac{\alpha}{2}}}{(T-v)^{1-\frac{\alpha}{2}}}.
\end{align*} 
\end{proof}

\section{Numerical simulations}\label{sect:numerics}
This section deals with the algorithm used to compute
$(Y^n_{t_k},Z^n_{t_k})_{k=0,\cdots,n}$ and numerical experiments 
for three different terminal conditions. In each case the exact solution is
available and we are able to compute the error $(Y^n-Y,Z^n-Z)$ in $L_2$-norm.
\subsection{\bf Simulation of \texorpdfstring{$(\tau_1,\cdots,\tau_n)$}{tau} and \texorpdfstring{$B^n$}{Bn}}

In order to simulate $(\tau_1,\cdots,\tau_n)$, we use the fact that
\begin{align*}
 \tau_0=0 \qquad \mbox{ and } \qquad \forall k \ge 1,\qquad \tau_k=\tau_{k-1}+\sigma_k,
\end{align*}
where $(\sigma_k)_{1\le n}$ is an i.i.d. sequence whose common
law $\sigma$ represents the first exit time of the Brownian motion $B$ of the interval $[-\sqrt{h},\sqrt{h}]$,
\begin{align*}
  \sigma:=\inf\{t> 0 : |B_t|=\sqrt{h}\}
\end{align*}
From
the book of Borodin and Salminen \cite{BS_15}, we have that the Laplace transform of $\sigma$ is
given by $\E(e^{-\lambda \sigma})=\frac{1}{\cosh(\sqrt{2 \lambda h})}$.

Let $F$ denote the cumulative distribution function of $\sigma$. It holds
$\E(e^{-\lambda \sigma})=\lambda \hat{F}(\lambda)$, where
$\hat{F}$ is the Laplace transform of $F$. Then, to obtain $F$, it remains to
inverse numerically its Laplace transform. Once we have $F$, we simulate the
sequence $(\sigma_k)_{1\le k\le n}$ by following the steps of Algorithm \ref{algo1}.

\begin{algorithm}[H]
\caption{Simulation of the sequence $(\tau_1,\cdots,\tau_n)$}\label{algo1}
\begin{algorithmic} 
  \State  Simulate one vector with uniform law $(U_1,\cdots,U_n)$\\
  \State $\tau_0=0$
  \For {$k=1 :n$}
  \State Compute $\sigma_k:=F^{-1}(U_k)$
  \State Define $\tau_k=\tau_{k-1}+\sigma_k$
  \EndFor
\end{algorithmic}
\end{algorithm}

\subsection{\bf Simulation of \texorpdfstring{$B^n$}{bn}}
In order to get the trajectory $B^n_{t_1},\cdots,B^n_{t_n}$ ($B^n_{t_0}=0$), we simulate
an i.i.d. Bernoulli sequence $(\xi_k)_{1 \le k \le n}$ i.e. $\p(\xi_k=\pm 1)=\half$. Then
\begin{align}\label{eq26}
 B^n_{t_{k+1}}=\left\{ \begin{array}{ll}
    B^n_{t_k}+ \sqrt{h} & \mbox{ if } \xi_k=1\\
                         B^n_{t_k}-\sqrt{h} & \mbox{ otherwise. }\\
                       \end{array}
                                              \right.
\end{align}

\subsection{Simulation of \texorpdfstring{$(Y^n,Z^n)$}{yz}}
Since $B^n$ is built using the random walk \eqref{eq26}, it can be represented
by a recombining binomial tree. Both $(Y^n_{t_k})_{0 \le k \le n}$ and
$(Z^n_{t_k})_{0 \le k \le n-1}$ can then also be represented as a recombining
binomial tree. Since $Y^n_{t_n}=g(B^n_{t_n})$, we solve backward in time the
BSDE by following these equalities, ensuing from \eqref{discrete-BSDE-sums}
($Y^n_{t_k}$ has been replaced by $Y^n_{t_{k+1}}$ in the generator term, but
the error induced by this modification is smaller than the ones we consider)

\begin{align*}
  &Z^n_{t_k}=\frac{1}{\sqrt{h}}\E_{\tau_k}(Y^n_{t_{k+1}}\eps_{k+1}),\\
  &Y^n_{t_k}=\E_{\tau_k}(Y^n_{t_{k+1}}+h f(t_{k+1},B^n_{t_k},Y^n_{t_{k+1}},Z^n_{t_k})).
\end{align*}

\subsection{Study of the error \texorpdfstring{$\E|Y^n_{t_k}-Y_{t_k}|^2$}{y-y} and
  \texorpdfstring{$\E|Z^n_{t_k}-Z_{t_k}|^2$}{z-z}}
In this subsection we assume that we are able to compute the exact solution $(Y,Z)$.
We want to study numerically the convergence in $n$ of
$\E|Y^n_{t_k}-Y_{t_k}|^2$ and $\E|Z^n_{t_k}-Z_{t_k}|^2$, where $(Y,Z)$ solves
\eqref{BSDE} and $(Y^n,Z^n)$ solves \eqref{discrete-BSDE}. To do so, we approximate
the error $\E|A^n_{t_k}-A_{t_k}|^2$ ($A=Y$ or $A=Z$) by Monte Carlo:
\begin{align}\label{MC-error}
\E|A^n_{t_k}-A_{t_k}|^2\sim \frac{1}{M}\sum_{m=1}^M
|A^{n,m}_{t_k}-A^m_{t_k}|^2:=E_A
\end{align}

\begin{enumerate}
\item For each Monte Carlo simulation, we pick at random one sequence
  $(\xi_1,\cdots,\xi_n)$ (which gives the value of
  $(B^n_{t_1},\cdots,B^n_{t_n})$) and one sequence $(\tau_1,\cdots,\tau_n).$
\item From the sequence $(\xi_1,\cdots,\xi_n)$ we get the trajectory of $Y^n$, including $Y^n_{t_k}$.
\item From the sequence $(B_{\tau_1},\cdots,B_{\tau_n})$ (which is equal to  $(B^n_{t_1},\cdots,B^n_{t_n})$), we compute
  $B_{t_k}$ by using the Brownian bridge method. We deduce
  $(Y_{t_k},Z_{t_k})$ as functions of $B_{t_k}$.
\end{enumerate}

In the following experiments, we plot the logarithm of the errors $E_Y$ and $E_Z$
(defined in \eqref{MC-error}) w.r.t. $\log(n)$. From Theorem
\ref{the-Brownian-motion-result}, we get that  $\log(E_Y)$ and $\log(E_Z)$ decrease as
$-\frac{\alpha}{2}\log(n)$. By using a linear regression, we compute the
slope of the line solving the least square problem and compare it to $-\frac{\alpha}{2}$.

\subsection{Numerical Experiment}
\subsubsection{Case \texorpdfstring{$g(x)=e^{T+x}$}{T+x} and \texorpdfstring{$f(y,z)=y+z$}{y+z}}
We consider the BSDE with terminal condition $g(x)=e^{T+x}$ and driver
$f(y,z)=y+z$. In this case, we know that $Y_t=e^{T+B_t+\frac{5}{2}(T-t)}$.
We run $M=20000$ Monte Carlo simulations.

\begin{figure}[H]
  \centering
  \includegraphics[width=10cm]{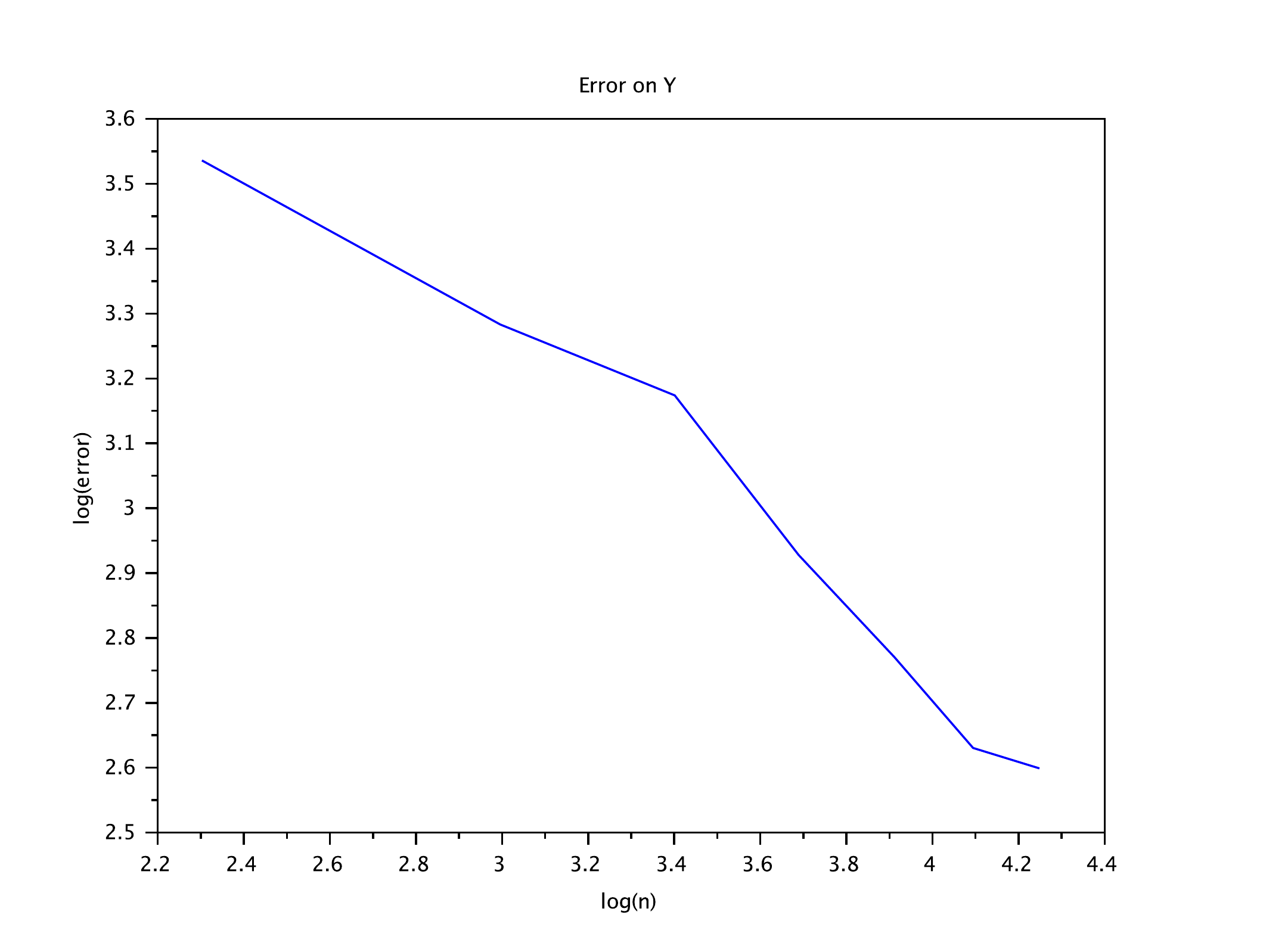} 
  \caption{$\log(\mbox{error on Y})$ w.r.t. $\log(n)$ - $f(y,z)=y+z$ -
    $g(x)=e^{T+x}$}
  \label{fig1}
\end{figure}

\begin{figure}[H]
  \centering
  \includegraphics[width=10cm]{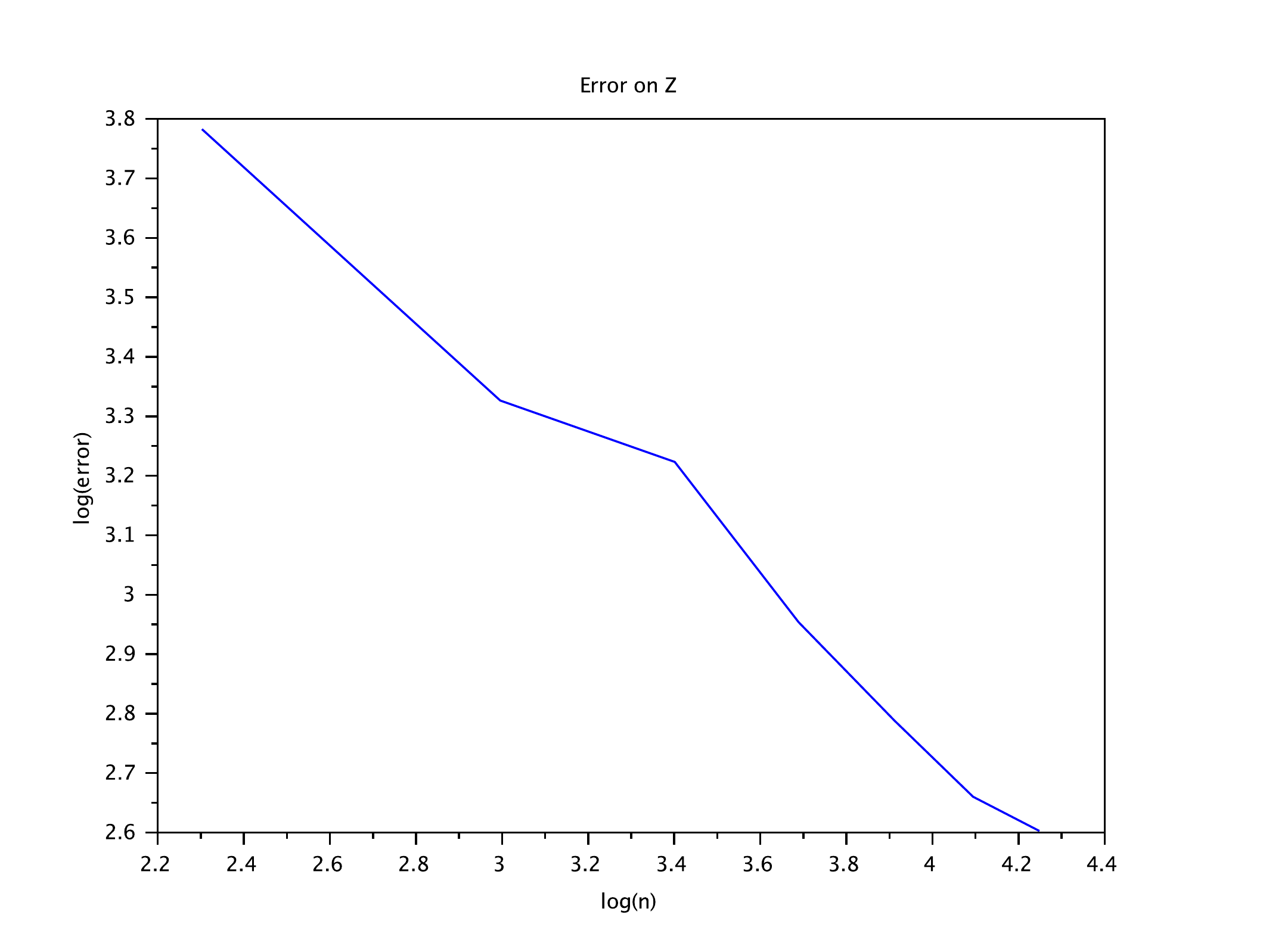}
  \caption{$\log(\mbox{error on Z})$ w.r.t.  $\log(n)$ - $f(y,z)=y+z$ -
    $g(x)=e^{T+x}$}
  \label{fig2}
\end{figure}

Figure \ref{fig1} (resp.~Figure \ref{fig2}) represents $\log(\mbox{error on Y})$ (the error is defined by
\eqref{MC-error}) (resp.~$\log$(error on Z)) with respect to
$\log(n)$. For the $Y$ case, the slope ensuing from the linear regression is $-0.53$. Even though  $g(x)=e^{T+x}$ does not satisfy \eqref{eq5}, 
$g$ is locally  Lipschitz continuous, and the outcome seems to be consistent
with Theorem \ref{the-Brownian-motion-result} for $\alpha=1$. For the $Z$ case, we get the slope $-0.61$.

\subsubsection{Case \texorpdfstring{$g(x)=x^2$}{x2} and \texorpdfstring{$f(y,z)=y+z$}{y+x}}
In that case, we know that $Y_t=e^{T-t}((B_t-(T-t))^2+ T-t)$ and
$Z_t=2e^{T-t}(B_t-(T-t))$.
We run $M=20000$ Monte Carlo simulations.

\begin{figure}[H]
  \centering
  \includegraphics[width=10cm]{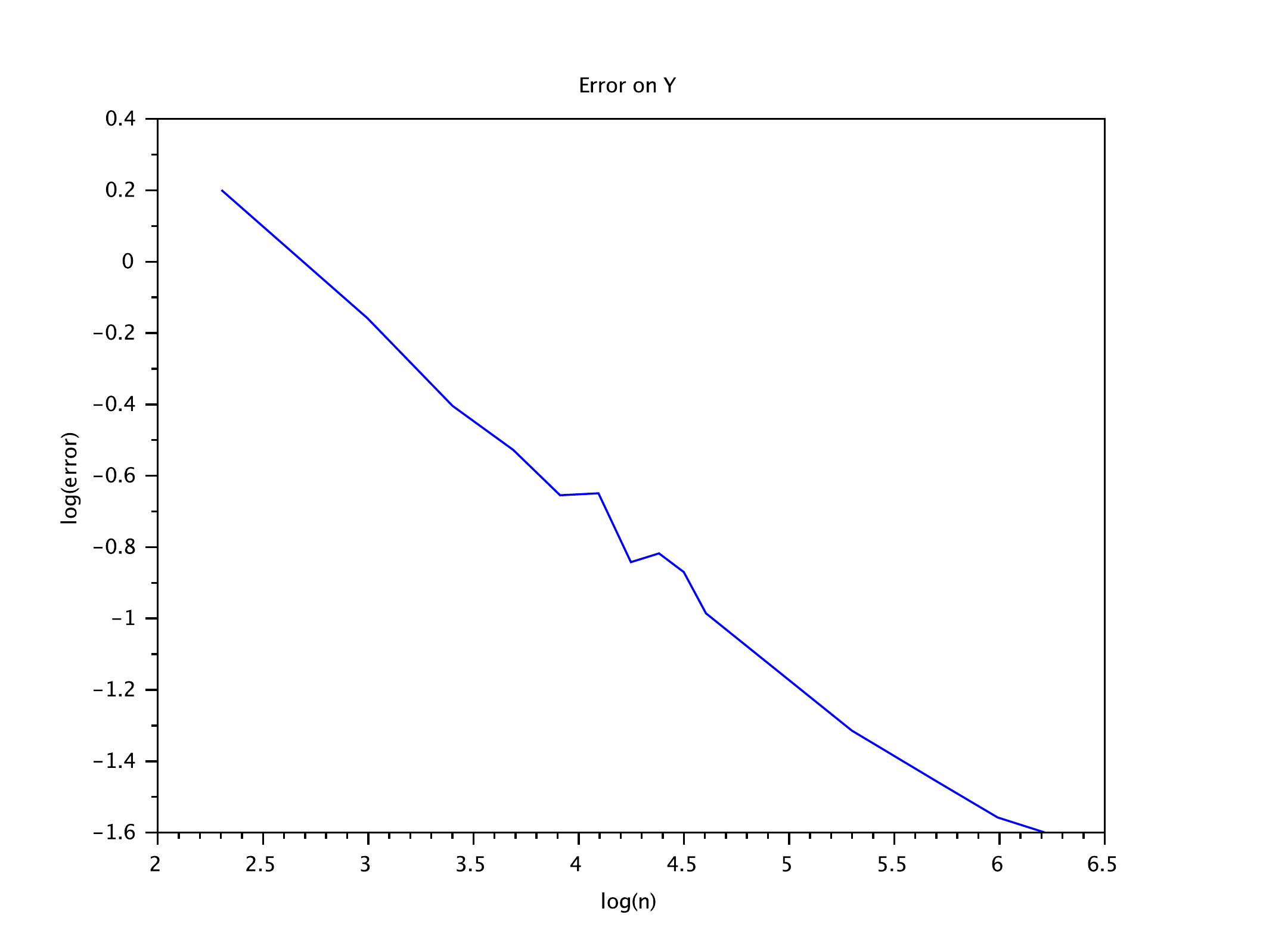}
  \caption{$\log(\mbox{error on $Y$})$ as a function of $\log(n)$ - $f(y,z)=y+z$ - $g(x)=x^2$}
  \label{fig3}
\end{figure}

\begin{figure}[H]
  \centering
  \includegraphics[width=10cm]{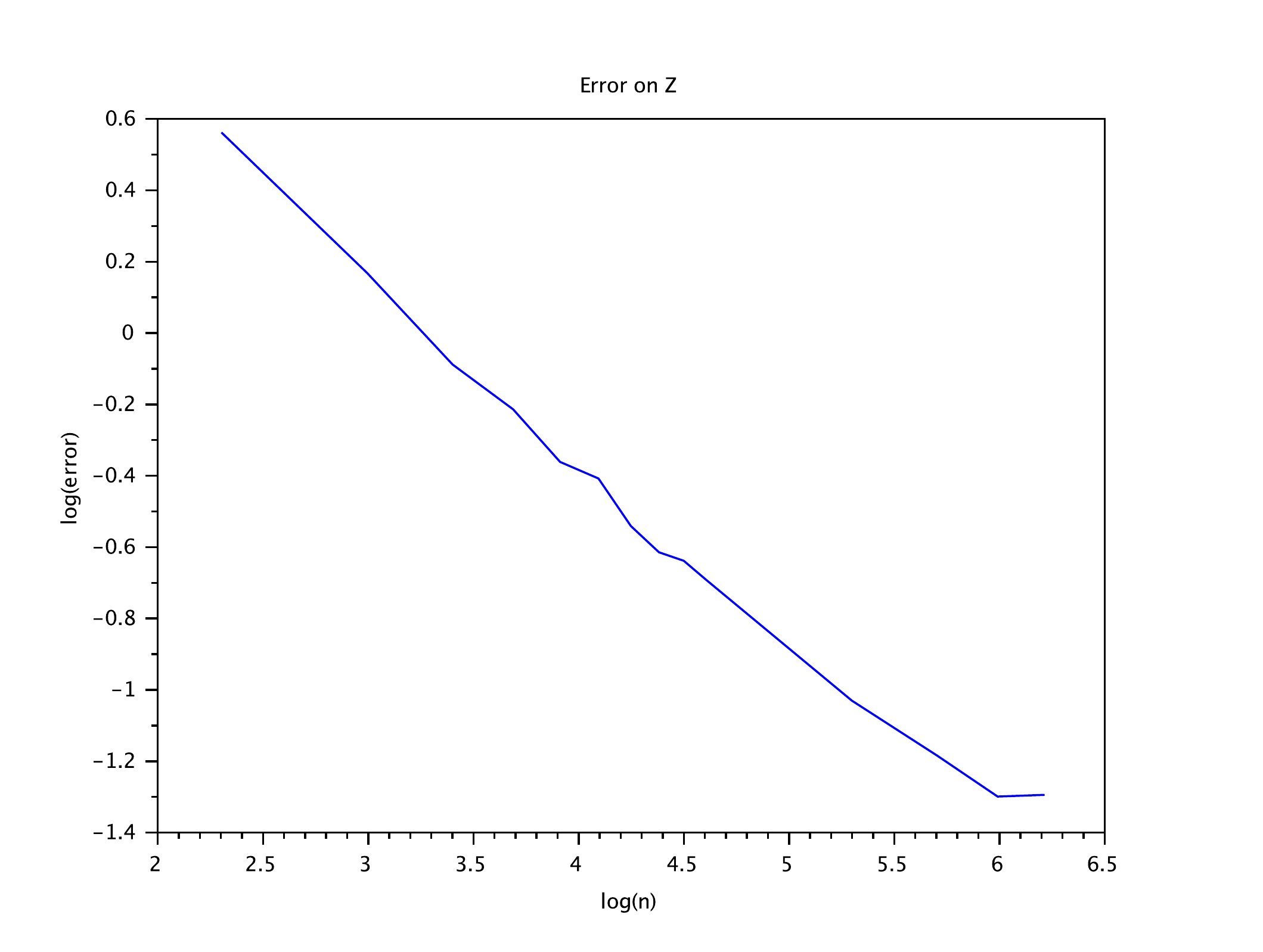}
  \caption{$\log(\mbox{error on $Z$})$ as a function of $\log(n)$ - $f(y,z)=y+z$ - $g(x)=x^2$}
  \label{fig4}
\end{figure}

Figure \ref{fig3} represents $\log(\mbox{error on $Y$})$ with respect to
$\log(n)$. The slope of the linear
regression is $-0.465$. Figure \ref{fig4} represents $\log(\mbox{error on $Z$})$ with respect to
$\log(n)$. The slope of the linear regression is $-0.48$. The results are then
consistent with Theorem \ref{the-Brownian-motion-result}.

\subsubsection{Case \texorpdfstring{$g(x)=\sqrt{|x|}$}{modulus} and \texorpdfstring{$f(y,z)=y+z$}{y+z}}
In that case, we know that $Y_t=e^{\frac{T-t}{2}}\ti \e(\sqrt{|\tilde{B}_{T-t}+B_t|}e^{\tilde{B}_{T-t}})$.
We run $M=20000$ Monte Carlo simulations.

\begin{figure}[H]
  \centering
  \includegraphics[width=10cm]{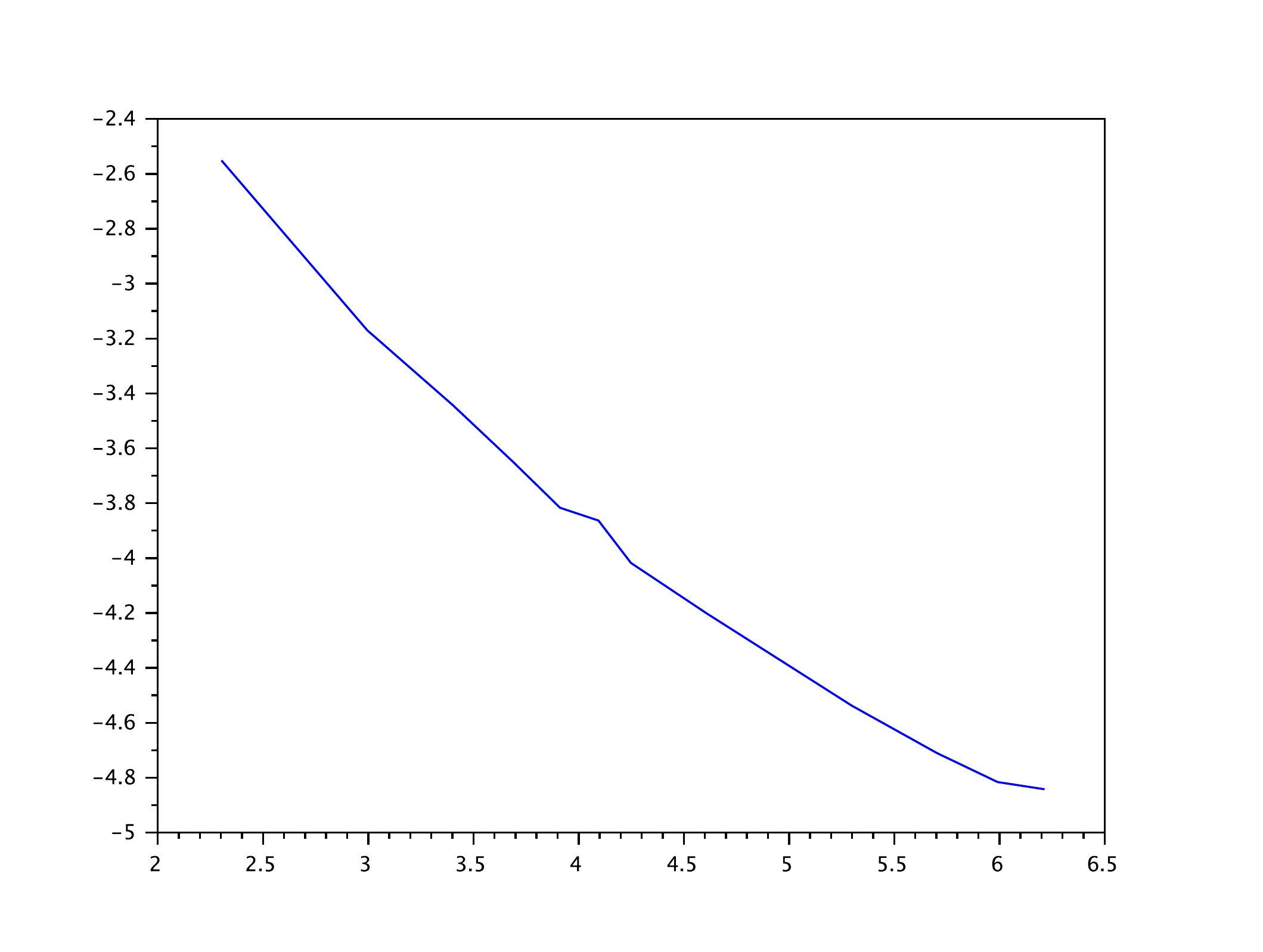}
  \caption{$\log(\mbox{error on $Y$})$ as a function of $\log(n)$ - $f(y,z)=y+z$ - $g(x)=\sqrt{|x|}$}
  \label{fig5}
\end{figure}

Figure \ref{fig5} represents $\log(\mbox{error on $Y$})$ with respect to
$\log(n)$. The slope of the linear
regression $-0.56$. Here we notice that the modulus of the slope we get is larger than
$\frac{1}{4}$, the upper bound obtained in that case in Theorem \ref{the-Brownian-motion-result}.

\section{Some properties of solutions to PDEs and BSDEs}\label{sect:PDE_FBSDE}
In the following we recall and prove results for FBSDEs with a general forward
process, even though we apply them in the present paper only for the case
where the forward process is just the Brownian motion. Restricting ourselves
to the case of Brownian motion would not shorten the proofs considerably.
Let us consider the following SDE started in $(t,x)$, 
\equal \label{process-X}
X^{t,x}_s = x +  \int_t^s b(r,X^{t,x}_r)dr + \int_t^s  \sigma(r, X^{t,x}_r)dB_r,   \quad 0\le t \le s \le T,
\tionl
where $b$ and $\sigma$ satisfy

\begin{hypo}\label{hypo2}\hfill
  \begin{enumerate}
\item   $b, \sigma  \in C_b^{0,2}([0,T]\times \R),$  in the sense that the  derivatives of order $k=0,1,2$  w.r.t.~the space variable  are continuous and bounded  on  $[0,T]\times \R,$
\item the  first and second derivatives of $b$  and  $\sigma$  w.r.t.~the space variable are assumed to be $\gamma$-H\"older continuous 
(for some $\gamma \in (0,1],$ w.r.t. the parabolic metric
$d((x,t),(x',t'))=(|x-x'|^2 +|t-t'|)^\half$ on all compact subsets of $[0, T ] \times \R$, 
\item   $b, \sigma$ are $\half$-H\"older continuous in time,
uniformly in space, 
\item $\sigma(t,x) \ge \delta >0$ for all $(t,x).$ 
\end{enumerate}
\end{hypo}

\subsection{Malliavin weights}

In this section we recall the Malliavin weights and their properties from  \cite[Subsection 1.1 and Remark 3]{GGG}. 
\begin{lemme}  \label{Malliavin-weights} Let  $H: \R \to \R$ be a polynomially bounded Borel function.
If  Assumption \ref{hypo2} holds and  $X^{t,x}$   is given by \eqref{process-X}, 
then setting  $$  G(t,x) := \E H(X_R^{t,x}) $$ 
implies that 
 $ G \in  C^{1,2}([0,R)\times \R ).$ 
Especially it holds for $0 \le t \le r <R \le T$  that 
$$ \partial_x G(r, X_r^{t,x}) =  \E [ H(X_R^{t,x}) N_R^{r,1,(t,x)} |\cf^t_r ], \quad \text{ and } \quad 
\partial^2_x G(r, X_r^{t,x}) = \E [ H(X_R^{t,x}) N_R^{r,2,(t,x)} |\cf^t_r ], $$
where $(\cf^t_r)_{r\in [t,T]}$  is the augmented natural filtration of  $(B^{t,0}_r)_{r \in [t,T]},$
$$N_R^{r,1,(t,x)}= \frac{1}{R-r} \int_r^R\frac{\nabla X^{t,x}_s }{\sigma(s,X_s^{t,x}) \nabla X^{t,x}_r  } dB_s \,\, \text{and} \,\,
N_R^{r,2,(t,x)}=\frac{ N^{\rho,1 ,(t,x)}_R \nabla X_R^{t,x}  N^{r,1 ,(t,x)}_\rho + \nabla N^{\rho,1 ,(t,x)}_R}{\nabla X^{t,x} _r},$$
with $\rho:=\frac{r+R}{2}$. Moreover,  for $ q \in (0, \infty)$  
 it holds a.s.
\equal \label{norm-weight}
 (\e[| N_R^{r,i,(t,x)}|^q |\cf^t_r ])^\frac{1}{q} \le \frac{\kappa_q}{(R-r)^\frac{i}{2}}, 
 \tionl
and  $\e[ N_R^{r,i,(t,x)} |\cf^t_r ]=0$  a.s. for $i=1,2.$ Finally, we have 
  $$\|   \partial_x G(r, X_r^{t,x}) \|_{L_p(\PP)} \le \kappa_{q} \frac{\|H(X_R^{t,x}) - \e[ H(X_R^{t,x})|\cf^t_r ]\|_{L_p(\PP)}  }{\sqrt{R-r}}   $$
  and
  $$\|   \partial^2_x G(r, X_r^{t,x}) \|_{L_p(\PP)} \le \kappa_{q} \frac{\|H(X_R^{t,x}) - \e[ H(X_R^{t,x})|\cf^t_r ]\|_{L_p(\PP)}  }{R-r}   $$
for $1<q,p < \infty$ with $\frac{1}{p} + \frac{1}{q} =1.$  
\end{lemme}

\subsection{Regularity of solutions to BSDEs}

Let us now consider the FBSDE
\begin{align}\label{BSDE2}
Y^{t,x}_s&= g( X^{t,x}_T) + \int_s^T f(r,X^{t,x}_r,Y^{t,x}_r,Z^{t,x}_r)dr - \int_s^T Z^{t,x}_r dB_r,  \quad \quad 0 \le t\leq s  \leq T, 
\end{align}
where  $X^{t,x}$ is the process satisfying \eqref{process-X}.
The following result is taken from \cite[Theorem 1]{GGG}.
We reformulate it here for the simple situation where we need it. On the other hand,
we will use $\PP_{t,x}$ and are interested in an estimate for all $(t,x) \in  [0,T) \times \R.$

\begin{thm} \label{difference-estimates for Y and Z}
Let Assumption \ref{hypo3} and  \ref{hypo2} hold.
Then for any $p\in [2,\infty)$ the following assertions are true.
\begin{enumerate}[(i)]
\item There exists a constant $C^y_{\ref{difference-estimates for Y and Z}} >0$ such that for $0\le t < s<T$ and $x\in \R,$
\equa
 \| Y_s - Y_t\|_{L_p(\PP_{t,x})} \le C^y_{\ref{difference-estimates for Y and Z}}  \Psi(x) \left ( \int_t^s (T-r)^{\alpha -1}dr \right )^\half,
\tion

\item There exists a constant $C^z_{\ref{difference-estimates for Y and Z}} >0$ such that for $0\le t <  s<T$ and $x\in \R,$
\equa
 \| Z_s - Z_t\|_{L_p(\PP_{t,x})} \le C^z_{\ref{difference-estimates for Y and Z}}  \Psi(x)  \left ( \int_t^s (T-r)^{\alpha -2}dr \right )^\half.
\tion
\end{enumerate}
The constants $C^y_{\ref{difference-estimates for Y and Z}}$ and
$C^z_{\ref{difference-estimates for Y and Z}}$ depend on $ K_f, L_f, C_g, c^{1,2}_{\ref{thm1}},T, p_0, b,\sigma, \kappa_q$ and $p$.
\end{thm}

\begin{proof}[Proof of Theorem \ref{difference-estimates for Y and Z}] (i)
First we follow  the step \cite[Theorem 1, proof of   $(C2_l) \implies (C3_l)$]{GGG}.
   We conclude  from the linear growth $ |f(r, x,y,z)| \le L_f (|x|+|y| + |z| )+  K_f $ 
      and from the Burkholder-Davis-Gundy
   inequality with constant $a_p>0$  that
\equa
&& \les \| Y_s - Y_t\|_{L_p(\PP_{t,x})}  \\&=& \left \| \int _t^s f(r,X_r, Y_r,Z_r)dr - \int _t^s Z_rdB_r \right \|_{L_p(\PP_{t,x})} \\
&\le & K_f (s-t) + L_f \int _t^s \| X_r  \|_{L_p(\PP_{t,x})}+  \|   Y_r  \|_{L_p(\PP_{t,x})}+  \| Z_r\|_{L_p(\PP_{t,x})}  dr  + a_p \left ( \int _t^s \|Z_r  \|^2_{L_p(\PP_{t,x})} dr \right )^\half.
\tion
We then use (i) and (ii) of Theorem \ref{thm1}  below to get
\equa
&& \les \| Y_s - Y_t\|_{L_p(\PP_{t,x})} \\
&\le & K_f (s-t) + C(T, L_f,c^{ {1,2}}_{\ref{thm1}}, p,  b,\sigma, p_0) \Psi(x) \bigg[ \int _t^s \Big(1 + {(T{-}r)^{\frac{\alpha-1}{2}}}\Big) dr  + \left (\int _t^s (T{-}r)^{\alpha-1} dr \right )^\half\bigg].
\tion
 (ii) Here one can follow  \cite[Theorem 1, proof of   $(C4_l) \implies (C1_l)$]{GGG}. \\
 {\bf Step 1:} We first assume additionally  that  $f:[0,T]\times \R^3
  \to \R$ is  continuously differentiable in $x,$ $y$, and $z$ with uniformly
  bounded derivatives as it was assumed for 
  \cite[Theorem 1]{GGG}. To take the dependency on $x$ into consideration which arises since we
 use  $\PP_{t,x}, $ it suffices to replace everywhere in the proof in  \cite{GGG} the constant $c_{B^\Theta_{p, \infty}}$ by $C(T,C_g, \sigma,b, p,p_0) \Psi(x).$ The constant $C^z_{\ref{difference-estimates for Y and Z}} $  depends moreover  on $L_f$ and $ \kappa_q.$  \\
 {\bf Step 2:}  Now let $f$ be as in Assumption \ref{hypo2}.
 In \cite[Theorem 1, proof of   $(C4_l) \implies (C1_l)$]{GGG} a linear BSDE is used
 which describes the behaviour of  the  process Z minus its counterpart where the generator is identically $0.$ Here the partial derivatives of $f_x, f_y,f_z$  appear but only their uniform bound 
 is needed in the estimates. Hence if $f$  satisfies \eqref{Lipschitz}, we can use mollifying as  explained  in  \eqref{mollified} below (one may choose $N=\infty$).  Since
    $|\partial_x  f^\eps(t,x,y,z)|,|\partial_y  f^\eps(t,x,y,z)|$ and $|\partial_z  f^\eps(t,x,y,z)|$ are bounded by  $ L_f$ we conclude from Step 1 that for all $\eps >0$  the process $Z^\eps$ corresponding to $ f^\eps$  
 satisfies 
 \equal \label{eps-estimate}
 \| Z^\eps_s - Z^\eps_t\|_{L_p(\PP_{t,x})} \le C^z_{\ref{difference-estimates for Y and Z}}  \Psi(x)  \left ( \int_t^s (T-r)^{\alpha -2}dr \right )^\half
\tionl
 for $p\ge 2.$    Especially, the family  $\{|Z^\eps_s - Z^\eps_t|^q: \eps>0  \}$  is then uniformly integrable provided that $q<p.$ By an a priori estimate  (cf.  \cite[Lemma 3.1]{Briand-DHPS}) we have that 
    $$ \e \int_0^T |Z_r-Z^\eps_r|^2 dr \le C    \int_0^T \sup_{x,y,z}
    |f(r,x,y,z)- f^\eps(r,x,y,z)|^2dr \le C  \eps^2 T L_f^2.  $$ 
Fubini's theorem implies that there exists  a sequence $\eps_m \to 0$ and a measurable set $N \subseteq [0,T]$ of Lebesgue measure zero, such that $\lim_{m \to \infty} \e|Z_r-Z^{\eps_m}_r|^2 =0 $   for all $r \in [0,T]\setminus N. $ Consequently, for any $q < p$  and all $t,s \in [0,T]\setminus N$ with $t<s,$
   \equa
 \| Z_s - Z_t\|_{L_q(\PP_{t,x})} \le C^z_{\ref{difference-estimates for Y and Z}}  \Psi(x)  \left ( \int_t^s (T-r)^{\alpha -2}dr \right )^\half.
 \tion  The assertion follows for all $q\ge 2$ since \eqref{eps-estimate} holds for all $p\in [2,\infty).$  Since by  Theorem \ref{thm1} (ii) the process $Z$ does have a continuous version, we finally 
 get the assertion for all $t<s.$ 
\end{proof}

\subsection{Properties of the associated PDE}
We collect in the theorem below properties of the solution to the PDE which
are mainly known. The new part concerns $\partial^2_x u$. For Lipschitz continuous $g,$  the behaviour of $\partial^2_x u$  has been studied in \cite{ZhangIII}. General results
related to this topic can be found in \cite{Crisan}.
\begin{thm}\label{thm1}
  Consider the FBSDE \eqref{BSDE2} and let  Assumptions \ref{hypo3}  and
  \ref{hypo2} hold. 
  Then for the solution $u$ of the associated PDE
 \equa
\left\{ \begin{array}{l}  u_t(t,x) +   \tfrac{\sigma^2(t,x)}{2} u_{xx}(t,x) + b(t,x) u_x(t,x) + f(t,x,u(t,x), \sigma(t,x) u_x(t,x)) =0,\\
 \hspace{30em} t\in [0,T), x\in \R, \\
 u(T,x)=g(x) , \quad x \in \R \end{array}\right . 
 \tion
   we have
  \begin{enumerate}[(i)]
  \item \label{first} $Y_t=u(t,X_t)$ where $u(t,x)=\e_{t,x} \left (g(X_T)+\int_t^T f(r,X_r,Y_r,Z_r)dr \right )$
    and $|u(t,x)|\le c^{1}_{\ref{thm1}} \Psi(x)$
    with $\Psi$ given in \eqref{Psi}, where $c^{1}_{\ref{thm1}}$ depends  on $C_g,T,p_0, L_f,K_f$ and on the bounds and  Lipschitz constants of $b$ and $\sigma.$
    \item  \label{second}  $u_x$ exists,
   \begin{align}\label{dxu}
      u_x(t,x) &=\e_{t,x}\left(g(X_T) N^{t,1}_T +\int_t^T
                          f(r,X_r,Y_r,Z_r)N^{t,1}_r dr\right), 
    \end{align}
 and
    \begin{enumerate}
    \item $u_x$ is continuous in $[0,T)\times \R,$ 
    \item  $Z^{t,x}_s= u_x(s,X_s^{t,x})\sigma(s,X_s^{t,x})$,
    \item $|u_x(t,x)|\le
      \frac{c^{2}_{\ref{thm1}} \Psi(x)}{(T-t)^{\frac{1-\alpha}{2}}}$,
    \end{enumerate}
    where $c^{2}_{\ref{thm1}}$ depends on $C_g,T,p_0,\kappa_2,L_f, K_f$   and on the bounds and  Lipschitz constants of $b$ and $\sigma.$  
  \item \label{third} $ u_{xx}$ exists,
    \begin{align}\label{d2xu}
       u_{xx}(t,x)&=\e_{t,x}\Bigg ( g(X_T) N^{t,2}_T  +\int_t^T [ f(r,X_r,Y_r,Z_r) -  f(r,X_t,Y_t,Z_t)]N^{t,2}_r dr\Bigg),
    \end{align}
    and
    \begin{enumerate}
    \item $ u_{xx}$ is continuous in $[0,T)\times \R,$
    \item $| u_{xx}(t,x)|\le
      \frac{ c^{3}_{\ref{thm1}} \Psi(x)}{(T-t)^{1-\frac{\alpha}{2}}},$
    \end{enumerate}
    where $c^{3}_{\ref{thm1}}$ depends on $C_g,T,p_0,\kappa_2,L_f,C^y_{\ref{difference-estimates for Y and Z}},C^z_{\ref{difference-estimates for Y and Z}}$ and on the bounds and  Lipschitz constants of $b$ and $\sigma.$
  \end{enumerate}
   In the following $c_{\ref{thm1}}$ represents
  $(c^{1}_{\ref{thm1}},c^{2}_{\ref{thm1}},c^{3}_{\ref{thm1}})$ and
  $c^{i,j}_{\ref{thm1}}$ ($i\neq j$) represents
  $(c^{i}_{\ref{thm1}},c^{j}_{\ref{thm1}})$, $(i,j)\in\{1,2,3\}$.
\end{thm}

\begin{proof} \eqref{first}: This follows from \cite[Theorem 3.2]{ZhangII}. \\
\eqref{second}: From the proof of \cite[Theorem 3.2]{ZhangII}, we get \eqref{dxu}. 
 The points \eqref{second}$(a)$ and $(b)$ ensue from \cite[Theorem 3.2 (i)]{ZhangII}. It remains to prove  $(c)$.\\
{\bf Proof of  \eqref{second} (c):} We show the assertion for a generator not depending on $X,$ since the terms arising from that dependency would be easy to treat. 
Since  $ \e_{t,x} ( \e_{t,x} (g(X_T)) N_T^{t,1})=0$  we can subtract it from the right hand side of \eqref{dxu} and get 
    \begin{align*}
      \partial_x
      u(t,x)=\e_{t,x}\left( [g(X_T)-\e_{t,x} (g(X_T))]  N_T^{t,1} +\int_t^T
      f(r,Y_r,Z_r)  N_r^{t,1} dr\right ).
    \end{align*}
It holds    
    \begin{align*}
 \E_{t,x} | g(X_T) - \e_{t,x} g(X_T)|^2 
 &= \e_{t,x} |g(X_T) - \tilde \e g(\tilde X^{t,X_t}_T)|^2  \le  \e_{t,x}\tilde \e |g(X_T) -  g(\tilde X^{t,X_t}_T)|^2,
 \end{align*}
 and thanks to the Cauchy-Schwarz inequality  with $\Psi_1= C_g(1+|X_T  |^{p_0}+ |\tilde X^{t,X_t}_T |^{p_0})$ and equation \eqref{eq5},
\begin{align} \label{terminal-est}
 \e_{t,x}\tilde \e |g(X_T) -  g(\tilde X^{t,X_t}_T)|^2 &\le \e_{t,x} \tilde \e (\Psi_1^2 |X_T - \tilde X^{t,X_t}_T |^{2\alpha}) \notag \\
&\le  \left[\e_{t,x} \tilde \e  \Psi_1^{4} \right]^{\frac{1}{2}} \left[ \e_{t,x}\tilde \e |X_T - \tilde X^{t,X_t}_T |^{4\alpha} \right]^{\frac{1}{2}} \notag \\
 &\le C(C_g ,T,p_0, b,\sigma)  \Psi^2(x)  (T-t)^{\alpha}.
\end{align}
 Relation  \eqref{norm-weight}  and the Lipschitz  continuity of $f$ imply
    \begin{align}\label{eq1}
      |\partial_x
      u(t,x)|\le&\frac{C(C_g,T,p_0, \kappa_2,b,\sigma) \Psi(x)}{(T-t)^{\frac{1-\alpha}{2}}}  \notag \\
      &+C(L_f , K_f )  \e_{t,x} \int_t^T
      (1+  |u(r,X_r)| 
      +  |\partial_x u (r,X_r)\sigma(r,X_r)| )|N_r^{t,1}| dr.
    \end{align} 
    Since we have $|g(x)|\le \Psi(x) $, \cite[Theorem 3.2
    (ii)]{ZhangII} gives $| u(t,x)|\le c \Psi(x)$ and $|\partial_x
    u(t,x)|\le c\Psi(x)(T-t)^{-1/2}$, where c depends on $T, L_f , K_f, \kappa_2, b,\sigma$ and $p_0.$
      Hence inequality \eqref{eq1}
    becomes
    \begin{align*}
      |\partial_x
      u(t,x)| & \le \frac{C(C_g,T,p_0, \kappa_2,b,\sigma) \Psi(x)}{(T-t)^{\frac{1-\alpha}{2}}}+C(L_f , K_f,c, \sigma) \e_{t,x}\left(\int_t^T
                       \left(1+\Psi(X_r)+\frac{\Psi(X_r)}{(T-r)^\half} \right)|N_r^{t,1}| dr\right)\\
      & \le \frac{C(C_g,T,p_0, \kappa_2,b,\sigma) \Psi(x)}{(T-t)^{\frac{1-\alpha}{2}}}+C( T, L_f , K_f, \kappa_2,b,\sigma, p_0)
       \int_t^T
                       \frac{ \Psi(x)}{(T-r)^\half (r-t)^\half} dr \\
        &\le \frac{C(C_g,T,p_0, \kappa_2,  L_f , K_f, b,\sigma) \Psi(x)}{(T-t)^{\frac{1-\alpha}{2}}}.
    \end{align*}
\smallskip \\

\eqref{third}:  We start with an approximation of  $g$ and $f$ by smooth and bounded functions.   
Let $\phi$ be a
non-negative $C^{\infty}$ function with support $[-1,1],$  such that $\int_{\R} \phi(u)
du =1,$ and $\eps \in (0,1].$ For $N \in \nset$ let   $b_N :\R \to [-N-1,N+1]$ be a  monotone $C^{\infty}$ function such that  $0\le b'_N(x) \le 1$ and
\equa
b_N(x) := \left \{ \begin{array}{cl} N+1, &\,\, x>N+2,\\
x,  & \,\, |x| \le  N, \\
-N-1, &\,\, x< -N-2.
\end{array} \right .
\tion
Define 
\equa
g^{\eps,N}(x)=\int_{-1}^1 \phi(u)  g(b_N(x)-\eps u) du
\tion
and 
\equal  \label{mollified}
    f^{\eps,N}(r,y,z)=\int_{-1}^1 \int_{-1}^1 \phi(u) \phi(v) f(r,b_N(y)-\eps u, b_N(z)-\eps v) du dv.
\tionl

    \begin{lemme}\label{lem1}
 $g^{\eps,N}$ and $f^{\eps,N}$ satisfy

 \begin{enumerate}[(a)]
   \item $\|g^{\eps,N}\|_{\infty} + \|f^{\eps,N}\|_{\infty} \le C= C(\eps, N)$ for some $C(\eps, N)>0,$
   \item $g^{\eps,N}$ and $f^{\eps,N}$  are $C^{\infty}$ functions, with bounded derivatives (the bounds depend on $\eps$ and $N$). Moreover, $f^{\eps,N}$ is a Lipschitz function in $y$ and $z$, with Lipschitz constant $L_f,$
   \item $g^{\eps,N}$ satisfies  \eqref{eq5}, uniformly in $\eps \in (0,1)$ and $N\ge 1,$
    \item for all $x \in \R$  and $\eps \in [0,1]$, we have $|g^{\eps,N}(x)-g(x)|\le  C(C_g) \Psi(x) (\eps^{\alpha}+\frac{|x|^{\alpha+1}}{N}),$
    \item for all $r \in [0,T]$ and for all $(y,z) \in \R^2$, we have $$|f^{\eps,N}(r,y,z) -f(r,y,z)|\le L_f(2\eps+|b_N(y)-y|+|b_N(z)-z|).$$
   \end{enumerate}
\end{lemme}

\begin{proof}
\begin{enumerate}[(a)]
  \item Since $g$ is locally Hölder continuous  in the sense of \eqref{eq5}, $|g(x)|\le C_g(1+|x|^{p_0+1}).$
  Then, we get $|g^{\eps,N}(x)|\le C_g(1+(N+1+\eps)^{p_0+1}),$  and for $f$ being Lipschitz continuous in $y$ and $z$, uniformly in time, the same type of result applies.
  \item Since $\phi$ is a $C^{\infty}$ function and $f$ and $g$ are of polynomial growth, we get the result.
  \item Since $g$ is locally Hölder continuous, we get
  \begin{align*}
    |g^{\eps,N}(x)-g^{\eps,N}(y)]
    &\le \int_{-1}^1 |\phi(u)| C_g(1 + |b_N(x)-\eps u|^{p_0}+ |b_N(y)-\eps u|^{p_0} )|b_N(x)-b_N(y)|^{\alpha}du\\
    &\le \int_{-1}^1C_g |\phi(u)|(1+(|x|+\eps)^{p_0}+ (|y|+\eps)^{p_0}  )|x-y|^{\alpha}du\\
    & \le C(C_g)(1+|x|^{p_0}+ |y|^{p_0} )|x-y|^{\alpha}.
  \end{align*}
  \item We have   
  \begin{align*}
    |g^{\eps,N}(x)-g(x)|&= \left |\int_{-1}^1\phi(u)(g(b_N(x)-\eps u)-g(x))du \right |\\
    & \le C_g \int_{-1}^1 |\phi(u)| (1+|b_N(x)|^{p_0}+\eps^{p_0}+|x|^{p_0})(|b_N(x)-x|^{\alpha}+\eps^{\alpha})du\\
    & \le C(C_g)(1+|x|^{p_0})(\eps^{\alpha}+|x|^{\alpha}\ind_{|x| \ge N}),
  \end{align*}
  and the result follows.
  \item We simply have to apply the Lipschitz property of $f$ to get the result.
\end{enumerate}
\end{proof}
\smallskip
  
 We put now $\eps:= \tfrac{1}{N}$ and write $(g^N,f^N)$ instead of $(g^{\frac{1}{N},N},f^{\frac{1}{N},N})$ in order to simplify the notation 
    and consider the BSDE
    \begin{align*}
      Y^N_t=g^N(X_T)+\int_t^T f^N(r,Y^N_r,Z^N_r)dr -\int_t^T Z^N_r dB_r.
    \end{align*}
{\bf Representation for  $\partial^2_x u^N(t,x).$} \\  By  \eqref{first}  we have that
 \begin{align*}
      u^N(t,x)&=\e_{t,x} g^N(X^{t,x}_T)  +\int_t^T\e_{t,x}   f^N(r,Y^N_r,Z^N_r)  dr.    \end{align*}
According to Lemma \ref{Malliavin-weights} it holds that  $\partial^2_x \, \e_{t,x} g^N(X_T)= \e_{t,x}  [ g^N(X_T)N_T^{t,2}] $ and
 $$\partial^2_x \,  \e_{t,x}  f^N(r,Y^N_r,Z^N_r) =  \e_{t,x} [  f^N(r,Y^N_r,Z^N_r)N_r^{t,2} ], $$
because    
$$ f^N(r, Y^N_r,Z^N_r) =  f^N(r,u^N(r,X_r),\sigma(r,X_r)  u_x^N(r,X_r)),$$
and $ f^N(r,y,z) $ is continuous and bounded.
   Moreover,    \cite[Proposition 4]{GobLab}  (or  \cite[Theorem 2.1]{Delarue_and_Menozzi}) implies that   $u^N(r,x)$
     is  $C^{1,2}$ and it  holds  that  $|u^N(r,x)|+|\partial_x u^N(r,x)|+|\partial_x^2
    u^N(r,x)|\le C^N$ for some $C^N>0.$  Since $\sigma$ is  continuous,  $$(r,x) \mapsto f^N(r, u^N(r,x),\sigma(r,x)  u_x^N(r,x))$$
    is a bounded Borel function.  Notice that by Lemma \ref{Malliavin-weights} 
 \begin{align} \label{E-of-N}
\e_{t,x} [N^{t,2}_r]=0 \quad  \text{and} \quad    \e_{t,x} [(N^{t,2}_r)^2] \le
   \frac{\kappa_2^2}{(r-t)^2},
\end{align}  
 so that   
$$ \e_{t,x} [ f^N(r, Y^N_r,  Z^N_r) N^{t,2}_r]= \e_{t,x} ([ f^N(r, Y^N_r,  Z^N_r)  -  f^N(r,Y^N_t,  Z^N_t)]N^{t,2}_r).$$   
Using the Lipschitz continuity of $f^N$ (see Lemma \ref{lem1}), the inequality of Cauchy-Schwarz  and  
Theorem \ref{difference-estimates for Y and Z} one can derive   the  upper  bound
\equal \label{upper-bound}
&& \less |\partial^2_x \,  \e_{t,x}  f^N(r,Y^N_r,Z^N_r)|   \notag \\
& \le& \e_{t,x} [ | f^N(r,Y^N_r,  Z^N_r)  -  f^N(r, Y^N_t,  Z^N_t)| |N^{t,2}_r|] \notag \\
&\le& C(L_f, \kappa_2)( \e_{t,x} (| Y^N_r -  Y^N_t | ^2+
|Z^N_r- Z^N_t) |^2))
^\half \,  \frac{1}{r-t} \notag  \\
&\le& C(L_f, \kappa_2,C^y_{\ref{difference-estimates for Y and Z}},C^z_{\ref{difference-estimates for Y and Z}})   \left   [  \left ( \int_t^r (T-s)^{\alpha-1}ds \right )^\half   
+  \left ( \int_t^r (T-s)^{\alpha-2}ds \right )^\half \right ] \, \frac{ \Psi(x)}{r-t}   \notag  \\
&\le& C(T,L_f, \kappa_2,C^y_{\ref{difference-estimates for Y and Z}},C^z_{\ref{difference-estimates for Y and Z}}) \Psi(x)\frac{1}{(T-r)^{1-\frac{\alpha}{2}} (r-t)^\half}.
\tionl
By this we do have an integrable bound for the derivative, and by dominated convergence we get   
 \equa  \partial^2_x \int_t^T\e_{t,x}
      f^N(r,Y^N_r,Z^N_r)    dr  &=&  \int_t^T\partial^2_x \e_{t,x}  f^N(r,Y^N_r,Z^N_r)   dr \\
 &=& \int_t^T       \e_{t,x} \{[ f^N(r, Y^N_r,  Z^N_r)  -  f^N(r, Y^N_tZ^N_t)]N^{t,2}_r\} dr.
 \tion     
Hence we can  write (using Fubini's theorem for the integral) 
\begin{align*}
      \partial^2_x u^N(t,x)&=\e_{t,x}\left(
      g^N(X_T) N^{t,2}_T +\int_t^T [ f^N(r, Y^N_r,  Z^N_r) - f^N(r, Y^N_t,  Z^N_t)]  N^{t,2}_r dr\right).
       \end{align*}
{\bf Convergence of $\partial^2_x u^N(t,x).$}
 Since   $\e_{t,x}[\e_{t,x}(g^N(X_T)) N^{t,2}_T ] =0,$  Cauchy-Schwarz's inequality  and the local Hölder continuity of $g^N$ (see Lemma \ref{lem1}) 
give like in \eqref{terminal-est} that
   \begin{align*}
 | \e_{t,x} (g^N(X_T)N^{t,2}_T)| &= \left  | \e_{t,x} \left([g^N(X_T)-\e_{t,x}(g^N(X_T))] N^{t,2}_T\right) \right | \\
      & \le \left (\e_{t,x} (|g^N(X_T)-\e_{t,x}(g^N(X_T))|^2 )\right )^\half \, \frac{\kappa_2}{T-t} \\
      &\le C(C_g, T, p_0,\kappa_2,b,\sigma)\frac{\Psi(x) }{(T-t)^{1-\frac{\alpha}{2}}},
   \end{align*}
  for all $N \in \nset.$
  For the second term we can use the upper bound \eqref{upper-bound} and Lemma
  \ref{beta-function} to get
  \begin{align*}
   \e_{t,x} \int_t^T\left | [ f^N(r, Y^N_r,  Z^N_r) -
    \right. & \left. f^N(r, Y^N_t,  Z^N_t)]  N^{t,2}_r\right | dr \notag\\
     & \le  C(T, L_f, \kappa_2 ,C^y_{\ref{difference-estimates for Y and Z}},C^z_{\ref{difference-estimates for Y and Z}})  \int_t^T  \frac{\Psi(x)}{(T-r)^{1-\frac{\alpha}{2}}
       (r-t)^\half}dr,\notag\\
    &\le C(T, L_f, \kappa_2 ,C^y_{\ref{difference-estimates for Y and Z}},C^z_{\ref{difference-estimates for Y and Z}})\Psi(x) \frac{B(\frac{\alpha}{2},\half)}{(T-t)^{\half-\frac{\alpha}{2}}}, 
  \end{align*}
  which implies
  \begin{align}\label{eq14}
    |  \partial^2_x u^N(t,x) | & \le C(C_g, T, L_f, p_0, \kappa_2,C^y_{\ref{difference-estimates for Y and Z}},C^z_{\ref{difference-estimates for Y and Z}},  b,\sigma ) \frac{   \Psi(x)  }{(T-t)^{1-\frac{\alpha}{2}}}.
  \end{align}
 According to   \cite[Theorem 2.1]{Delarue_and_Menozzi}  $ \partial^2_x u^N(t,x) $ is continuous.
  Let
 \begin{align*}
      v(t,x)&:=\e_{t,x}\left(
      g(X_T) N^{t,2}_T +\int_t^T [ f(r, Y_r,  Z_r) - f(r, Y_t,  Z_t)]  N^{t,2}_r dr\right). 
    \end{align*}
We show that  for any $(t,x)  \in [0,T)\times \rset$ it holds   $\partial^2_x u^N(t,x) \to v(t,x)$  if  $N \to \infty,$
and that $v$ is continuous  on $[0,T)\times \rset$.  The idea to show continuity of $v$ is as follows:  If $(t_n,x_n) \to (t,x),$ then we may assume
that  we can find a $\delta >0$ such that  $x_n \in (x-\delta, x+ \delta)$ and $t_n \in (t-\delta, t+ \delta) \subseteq [0,T)$ for each sufficiently large $n.$ We consider
\equa
|v(t_n,x_n) - v(t,x)| &\le&  |v(t_n,x_n) - \partial^2_x u^N (t_n,x_n)| +
|\partial^2_x u^N (t_n,x_n) -\partial^2_x u^N (t,x)| \\&& + |\partial^2_x u^N (t,x) - v(t,x)|.
\tion
Since $\partial^2_x u^N$ is continuous, the term $|\partial^2_x u^N (t_n,x_n) -\partial^2_x u^N (t,x)|$ is small for large $n.$ 
Hence it  suffices to show that  $\sup_{s\in(t-\delta, t+ \delta), y \in (x-\delta, x+ \delta)}  |\partial^2_x u^N (s,y) - v(s,y)|$ is small for large $N.$  
Let $(s,y) \in  (t-\delta, t+ \delta)\times (x-\delta, x+ \delta).$ It holds
\begin{align*}
&|\partial^2_x u^N(s,y) - v(s,y)| \le \e_{s,y}| [g^N(X_T)   - g(X_T) ]N^{s,2}_T  |+  \int_s^T D^\half(r,s) \frac{\kappa_2}{r-s}dr := D_1+D_2,
\end{align*}
where (setting $\|\cdot\|_{\p_{s,y}}:= \|\cdot\|_{L_2(\p_{s,y})}$)
\equa
D(r,s)&:=&  \|  f^N(r, Y^N_r,  Z^N_r) - f^N(r, Y^N_s,  Z^N_s)
- [ f(r, Y_r,  Z_r) - f(r, Y_s,  Z_s)] \|^2_{\p_{s,y}} \\
&\le&  L_f ( \|Y^N_r- Y^N_s \|_{\p_{s,y}}   +\|  Z^N_r- Z^N_s \|_{\p_{s,y}} + \|Y_r- Y_s \|_{\p_{s,y}}  +\|  Z_r- Z_s \|_{\p_{s,y}} ) \\
&&\times (\| f^N(r, Y^N_r,  Z^N_r) - f(r, Y_r,  Z_r)  \|_{\p_{s,y}} + \|  f^N(r, Y^N_s,  Z^N_s) - f(r, Y_s,  Z_s)  \|_{\p_{s,y}}).
\tion
First, let us bound $D_1$.  According to Cauchy-Schwarz's  inequality, \eqref{eq10}  below and \eqref{E-of-N} we get
\begin{align*}
  D_1 \le \delta_1\sqrt{\e_{s,y}(|N^{s,2}_T|^2)}\le \frac{\delta_1\kappa_2}{T-s}\ \le
  \frac{\delta_1 \kappa_2}{T-t-\delta}.
\end{align*}
Now let us bound $D_2$. According to Theorem \ref{difference-estimates for Y and Z} 
 it holds
\begin{align*}
D^\half(r,s)\le& C( T,L_f,C^y_{\ref{difference-estimates for Y and Z}},C^z_{\ref{difference-estimates for Y and Z}})  \Psi^{\half}(y)   \frac{(r-s)^\frac{1}{4}}{       (T-r)^{\half-\frac{\alpha}{4}} } \\
&\times(\| f^N(r, Y^N_r,  Z^N_r) - f(r, Y_r,  Z_r)
  \|_{\p_{s,y}}  + \|  f^N(r, Y^N_s,  Z^N_s) - f(r, Y_s,  Z_s)  \|_{\p_{s,y}})^\frac{1}{2} .\\
\end{align*}
Then, using \eqref{eq8}, \eqref{eq12}, \eqref{eq11} and Proposition \ref{prop1} below gives
\begin{align*}
D^\half(r,s)\le& C(T, L_f, \kappa_2,C^y_{\ref{difference-estimates for Y and Z}},C^z_{\ref{difference-estimates for Y and Z}}) \Psi(y)   \frac{(r-s)^\frac{1}{4}}{    (T-r)^{\half-\frac{\alpha}{4}}   }     \frac{\delta_1 }{(T-r)^{\frac{1}{4}}}.
\end{align*}

Hence we have shown that
\begin{align*}
  D_2 &\le  C(T, L_f ,\kappa_2,C^y_{\ref{difference-estimates for Y and Z}},C^z_{\ref{difference-estimates for Y and Z}}) \Psi(y) \int_s^T\frac{\delta_1}{(r-s)^{\frac{3}{4}}(T-r)^{\frac{1}{2}-\frac{\alpha}{4}+\frac{1}{4}}} dr \\
 &\le C(T, L_f ,\kappa_2 ,C^y_{\ref{difference-estimates for Y and Z}},C^z_{\ref{difference-estimates for Y and Z}}) \Psi(y)  \frac{\delta_1}{(T-s)^{\frac{1}{2}-\frac{\alpha}{4}}} ,\\
 &\le C(T, L_f ,\kappa_2 ,C^y_{\ref{difference-estimates for Y and Z}},C^z_{\ref{difference-estimates for Y and Z}})  \Psi(x+\delta)  \frac{\delta_1}{(T-t-\delta)^{\frac{1}{2}-\frac{\alpha}{4}}} \quad   \forall  (s,y) \in  (t-\delta, t+ \delta)\times (x-\delta, x+ \delta).
\end{align*}
Consequently,  $\sup_{y \in (x-\delta, x+ \delta), s\in(t-\delta, t+ \delta)}  |\partial^2_x u^N (s,y) - v(s,y)|$ is small for  large  $N,$  hence $v$ is continuous.
Since
$$
  \partial_x u^N(t,x) -   \partial_x u^N(t,y) = \int_y^x \partial^2_x u^N(t,z) dz
$$
converges to
$$
  \partial_x u(t,x) -   \partial_x u(t,y) = \int_y^x  v(t,z) dz,
$$
it follows that $\partial^2_x u(t,x) =v(t,x).$ Then point (iii-a) and \eqref{d2xu} are proved. Since $\partial_x^2 u^N$ converges to $v$ for  $N \to \infty$, we deduce 
point (iii-b) from \eqref{eq14}.
\end{proof}

 \begin{prop}\label{prop1} Let Assumptions \ref{hypo2}  and
  \ref{hypo3} hold.
      Then for any  $(s,y) \in  (t-\delta, t+ \delta)\times (x-\delta, x+ \delta)$  with $t + \delta <T$   and  $r$  such that $s\le r < T$ we have
     \begin{align*}
        \|Y^N_r- Y_r\|_{L_2(\PP_{s,y})} +\|Z^N_r- Z_r\|_{L_2(\PP_{s,y})} \le \frac{\delta_1}{\sqrt{T-r}},
     \end{align*}
where $\delta_1$ denotes a generic constant which tends to $0$ when $N$ tends to $+\infty$.
   \end{prop}
\begin{proof}
   Let here  $\| \cdot \|$ stand for  $\| \cdot \|_{L_2(\PP_{s,y})}.$  We will use for the  $Y$ differences the  inequality
     \equa   \|Y^N_r- Y_r\|  \le \|g^N(X_T)-g (X_T) \|
        +      \int_r^T  \| f^N(w,Y^N_w,Z^N_w)-  f(w,Y_w,Z_w)\|  dw.
       \tion
     For the $Z$ differences we get by \eqref{dxu}  and  (\ref{second}-b)
     \equa
     && \les \|Z^N_r- Z_r\| \\
 &\le&  C(\sigma) \left (\left \| \e_r \left(g^{N }(X_T)- g(X_T) \right) N_T^{r,1} \right \|  
      +  \left \| \e_r  \int_r^T (f^N(w,Y^N_w,Z^N_w)-  f(w,Y_w,Z_w))N_w^{r,1}dw\right \|\right ) \\
       &\le&  C(\kappa_2,\sigma)  \left ( \frac{\|g^N(X_T)-g (X_T) \|}{\sqrt{T-r}  }  
        +     \int_r^T  \| f^N(w,Y^N_w,Z^N_w)-  f(w,Y_w,Z_w)\|    \frac{1}{\sqrt{w-r}} dw\right ). 
        \tion
Let $S(r):=\|Y^N_r- Y_r\| +   \|Z^N_r- Z_r\|$. Using the inequality $ (1+  \frac{1}{\sqrt{w-r}}) \le C(T)   \frac{1}{\sqrt{w-r}}$   for $r < w \le T$ gives
\begin{align}
  S(r)\le \,& C(T,\kappa_2, \sigma) \left( \|g^N(X_T)-g (X_T) \|  \frac{1}{\sqrt{T-r}} + \int_r^T  \| f^N(w,Y^N_w,Z^N_w)-  f(w,Y_w,Z_w)\| \frac{1}{\sqrt{w-r}} dw\right ). \label{eq9}
\end{align}
Let us bound $\|g^N(X_T)-g (X_T) \|$. By  Lemma \ref{lem1} we get the estimate
\begin{align}
 \e_{s,y}| g^N(X_T)  - g(X_T)|^2 & \le C(C_g)   \e_{s,y} \left (\Psi(X_T)^4 \right)^\half   \left(\e_{s,y} \left(\frac{1}{N^{\alpha}}+ \frac{|X_T|^{\alpha +1}}{N}\right)^4 \right)^\half \notag\\
  & \le C(C_g,T ,b,\sigma,p_0) \Psi(y)^2 \left  (\frac{1}{N^{2\alpha}}+  \frac{|y|^{2\alpha +2}}{N^2}  \right)  \notag\\
  & \le C(C_g,T ,b,\sigma,p_0)\Psi(x+\delta)^2 \left  (\frac{1}{N^{2\alpha}} +  \frac{|x+\delta|^{2\alpha +2}}{N^2} \right )\le \delta^2_1,
  \label{eq10}
  \end{align}
for any arbitrarily small $\delta_1>0,$ provided that $N$ is sufficiently large. 
  Let us now bound  $\| f^N(w,Y^N_w,Z^N_w)-  f(w,Y_w,Z_w)\|$. Using again Lemma \ref{lem1} yields to 
\begin{align}
&\less  \| f^N(w,Y^N_w,Z^N_w)-  f(w,Y_w,Z_w)\| \notag\\
 \le&  \| f^N(w,Y^N_w,Z^N_w)-  f^N(w,Y_w,Z_w) \|  +  \| f^N(w,Y_w,Z_w)-  f(w,Y_w,Z_w)\|   \notag\\
 \le&   L_f (\| Y^N_w-Y_w\| + \|Z^N_w -Z_w \|  + \tfrac{2}{N}+  \|b_N(Y_w)- Y_w\| +  \|b_N(Z_w)- Z_w\|). \label{eq8}
\end{align}
Then, plugging \eqref{eq10} and \eqref{eq8} into \eqref{eq9} gives
       \begin{align}\label{eq13}
      S(r)  \le  &  \frac{C(T,\kappa_2,\sigma)\delta_1}{\sqrt{T-r}} +    C(T,\kappa_2,\sigma)   L_f \int_r^T    \frac{S(w)}{\sqrt{w-r}}  dw \notag\\&
      \quad+   C(T,\kappa_2 ,\sigma)   L_f \int_r^T    \frac{\tfrac{1}{N}  +  \|b_N(Y_w)- Y_w\|   +  \|b_N(Z_w)- Z_w\|}{\sqrt{w-r}}  dw.
      \end{align}
     To estimate  $ \|b_N(Z_w)- Z_w\|    $ we use    $ Z_w =  \sigma(w,X_w) u_x(w,X_w) $ and choose  a small $a >0$ such that $\beta:=\frac{(2+a)(1-\alpha)}{2} <1.$ Then
      \begin{align*}
      \|b_N(Z_w)- Z_w \|^2 = \e _{s,y}  |b_N(Z_w)- Z_w |^2 \ind_{|Z_w| \ge N} \le \frac{  \e _{s,y}| Z_w |^{2+a}  }{N^a} = \frac{ \e _{s,y} |\sigma(w,X_w)  u_x(w,X_w) |^{2+a}  }{N^a}.
\end{align*}
Using  Theorem \ref{thm1} (ii-c) yields
\begin{align}
    \e _{s,y}  |b_N(Z_w)- Z_w |^2  &\le    \frac{
                                     C(c^{2}_{\ref{thm1}},\sigma) \e_{s,y}
                                     \Psi(X_w)^{(2+a)}}{(T-w)^{\frac{(2+a)(1-\alpha)}{2}}
                                     N^a } \le \frac{C(T, p_0,
                                     c^{2}_{\ref{thm1}},\sigma,b)  \, \Psi(y)^{(2+a)}}{(T-w)^{\frac{(2+a)(1-\alpha)}{2}} N^a} \notag \\
     &\le  \frac{\delta_1}{(T-w)^{\frac{(2+a)(1-\alpha)}{2}}} , \quad   \forall  (s,y) \in  (t-\delta, t+ \delta)\times (x-\delta, x+ \delta). \label{eq12}
   \end{align}
   Similarly,
     \begin{align}\label{eq11}
     \e _{s,y}  |b_N(Y_w)- Y_w |^2   &\le   \frac{C(T,
                                       p_0,c^{1}_{\ref{thm1}} ,b,\sigma) \, \Psi(y)^{(2+a)}}{ N^a}  \le  \delta_1 , \quad   \forall  (s,y) \in  (t-\delta, t+ \delta)\times (x-\delta, x+ \delta).
   \end{align}
Plugging \eqref{eq12} and \eqref{eq11} into \eqref{eq13} gives
\begin{align*}
  S(r) & \le \frac{C(T,\kappa_2)\delta_1}{\sqrt{T-r}} +      C(T,\kappa_2)   L_f \int_r^T    \frac{S(w)}{\sqrt{w-r}}  dw   \\ &\quad +  C(T,\kappa_2)   L_f \int_r^T    \frac{\tfrac{1}{N}  +  \delta_1 }{\sqrt{w-r}} +\frac{\delta_1}{(T-w)^{\frac{(2+a)(1-\alpha)}{2}}\sqrt{w-r}} dw\\
  &\le  C(T,\kappa_2,L_f)   \left ( \frac{\delta_1}{\sqrt{T-r}} +       \int_r^T    \frac{S(w)}{\sqrt{w-r}}  dw \right ),
\end{align*}

where the last inequality comes from Lemma \ref{beta-function} ($\beta < 1$).
It remains to apply a version of Gronwall's Lemma (see e.g. \cite[Lemma 3.1]{SY_09}) to see that $ S(r) \le \frac{C(T,\kappa_2,L_f)\delta_1}{\sqrt{T-r}}.$  Since $C(T,\kappa_2,L_f)\delta_1$ becomes arbitrarily small for $N$ large, we will slightly abuse the notation and write $S(r) \le \frac{\delta_1}{\sqrt{T-r}}.$
   \end{proof}

   \appendix
   \section{Technical results and estimates}

\begin{lemme}\label{B-difference}
For all $0 \leq k \leq m \leq n$ and $p > 0$, it holds for $h = \tfrac{T}{n}$ that
\begin{enumerate}[(i)]
\item $\e{\tau_k}  = kh$, \label{tauk-expectation}
\item $\e |\tau_1 |^p \le C(p) h^p$, \label{tau1-expectation}
\item $\e | B_{\tau_m}-B_{\tau_k} |^2  = t_m- t_k,$ \label{sigma-expectation}
\item $\e | B_{\tau_k} - B_{t_k}|^{2p} \le C(p) \e |\tau_k - t_k|^p \le C(p) (t_k h)^{\frac{p}{2}}. \label{B-difference-estimatep}$
\end{enumerate}
\end{lemme}
\begin{proof} The strong Markov property of the Brownian motion implies that $(\tau_i -\tau_{i-1})_{i=1}^\infty$ is an i.i.d. sequence.
According to \cite[Proposition 11.1 (iii)]{Walsh}, we have that $\e \tau_1 = \frac{T}{n}$, and \eqref{tauk-expectation} follows.
Item \eqref{tau1-expectation} follows by \cite[Proposition 11.1 (iv)]{Walsh} and Jensen's inequality. To prove item \eqref{sigma-expectation}, recall that $(B_{\tau_i} - B_{\tau_{i-1}})_{i=1}^{\infty}$ is a centered i.i.d.~sequence with $\E (B_{\tau_i} - B_{\tau_{i-1}})^2 = h$, $i \geq 1$. \eqref{B-difference-estimatep}: The BDG inequality implies that for each $p > 0$,
\begin{align*}
\e | B_{\tau_k} {-} B_{t_k}|^{p} &= \E \abs{\int_{0}^{\tau_k \vee t_k} \!\!
  (\ind_{[0,\tau_k]}(r) {-} \ind_{[0,t_k]}(r)) dB_r}^{p} \!\!\\
  &\le C(p)\;\; \E \pr{\int_{0}^{\tau_k \vee t_k} \!\! \ind_{[0,\tau_k] \Delta [0,t_k]}(r) dr}^{p/2} \!\!\! = \E|\tau_k {-} t_k|^{p/2}.
\end{align*}
To prove the second inequality of \eqref{B-difference-estimatep}, a generalization of \cite[Proposition 11.1
(iv)]{Walsh}, we first assume that $p \geq 1$. Let us rewrite $\tau_k-t_k=\sum_{i=1}^k \eta_i$ where $(\eta_i)_{1\le i \le k}$ is an
i.i.d.~centered sequence of random variables distributed as $\tau_1-h$. Burkholder's and Hölder's inequalities, and finally item \eqref{tau1-expectation} yield
\begin{align*}
  \e |\tau_k  -t_k|^p  \le C(p)\; \e \Big( \sum_{i=1}^k \eta_i^2
  \Big)^{\frac{p}{2}} \le k^{\frac{p}{2}-1}\sum_{i=1}^k \e (\eta_i^p) \le C(p) (t_k h)^{\frac{p}{2}},
\end{align*}
which proves the claim for $p \geq 1$. The case $p < 1$ follows from this result by Jensen's inequality.
\end{proof}

\begin{lemme}\label{beta-function}
    For all $t \in [0,T)$ and for all $\alpha<1$, $\beta<1$  we have
    \begin{align*}
      \int_t^T \frac{1}{(T-r)^{\alpha}(r-t)^{\beta}}dr=\frac{1}{(T-t)^{\alpha+\beta-1}}B(1-\alpha,1-\beta),
    \end{align*}
    where $B$ denotes the beta function.
\end{lemme}

\bibliographystyle{plain}

\end{document}